\documentclass[amscd,amssymb,verbatim,11pt]{amsart}
\numberwithin{equation}{section}
\usepackage{epsfig}
\usepackage{multido}
\usepackage{color}
\usepackage{pstricks}
\usepackage{pst-all}
\usepackage{pst-math}
\usepackage{pst-func}
\usepackage{pst-3dplot}

\newtheorem{thm}{Theorem}[section]
\newtheorem{cor}[thm]{Corollary}
\newtheorem{lem}[thm]{Lemma}

\theoremstyle{definition}

\newtheorem{rem}[thm]{Remark}

\newif\ifShowLabels
\ShowLabelstrue
\newdimen\theight
\def\TeXref#1{
     \leavevmode\vadjust{\setbox0=\hbox{{\tt
            \quad\quad  {\small  \bf #1}}}%
     \theight=\ht0
     \advance\theight  by  \dp0
     \advance\theight  by  \lineskip
     \kern -\theight \vbox  to
     \theight{\rightline{\rlap{\box0}}%
      \vss}%
      }}%

    {\begin{thm}\label{#1} \ifShowLabels \TeXref{#1} \fi}%
    {\end{thm}}

    {\begin{def}\label{#1} \ifShowLabels \TeXref{#1} \fi}%
    {\end{def}}

    {\begin{lem}\label{#1} \ifShowLabels \TeXref{#1} \fi}%
    {\end{lem}}

    {\begin{cor}\label{#1} \ifShowLabels \TeXref{#1} \fi}%
    {\end{cor}}

\newcommand{\eqRef}[1]%
     {\ifShowLabels \TeXref{#1} \fi
      \begin{equation}\label{#1} }

\ShowLabelsfalse

\newcommand{\vsp}{\vskip 1em}

\newcommand{\NI}{\noindent}
\newcommand{\bea}{\begin{eqnarray}}
\newcommand{\eea}{\end{eqnarray}}
\newcommand{\bas}{\begin{align*}}
\newcommand{\eas}{\end{align*}}
\newcommand{\ba}{\begin{align}}
\newcommand{\ea}{\end{align}}

\newcommand{\be}{\begin{equation}}
\newcommand{\ee}{\end{equation}}
\newcommand{\ben}{\begin{eqnarray*}}
\newcommand{\een}{\end{eqnarray*}}
\newcommand{\lam}{\lambda}
\newcommand{\Om}{\Omega}

\newcommand{\tht}{\theta}
\newcommand{\p}{\partial}
\newcommand{\al}{\alpha}

\newcommand{\Lam}{\Lambda}
\newcommand{\g}{\gamma}
\newcommand{\vep}{\varepsilon}
\newcommand{\dl}{\delta}
\newcommand{\D}{\Delta}

\newcommand{\vt}{\vartheta}

\newcommand{\s}{\sigma}

\newcommand{\lamo}{\lambda_{\Om}}
\newcommand{\rh}{\rho}

\newcommand{\IR}{\mathbb{R}}

\title{On the viscosity solutions to some nonlinear elliptic equations}

\author[T. Bhattacharya and L. Marazzi]{Tilak Bhattacharya and Leonardo Marazzi}

\begin{document}

\maketitle

\begin{abstract} We consider viscosity solutions of a class of nonlinear degenerate elliptic equations on bounded domains. We prove comparison principles and a priori supremum bounds for the solutions.
We also address the eigenvalue problem and, in many instances, show the existence of a first eigenvalue and a first positive eigenfunction. 
\vsp
\NI AMS classification: 35J60, 35J70, 35P30\\
\NI Keywords: viscosity solutions, eigenvalue, degenerate elliptic
\end{abstract}

\section{Introduction and statements of main results}

In this work, we study issues related to the eigenvalue problem for some nonlinear degenerate elliptic operators. This may be considered as a follow-up of the work in \cite{BL},  where we showed the existence of the first eigenvalue and a positive first eigenfunction of the infinity-Laplacian. The current work continues the effort of studying similar questions for a more general class of nonlinear elliptic, possibly degenerate, operators. See \cite{BNV, BD, BK, BL, J, JLM, JL, QS}.
 
To state our results more precisely, we introduce notations that will be used through out this work. Let $\Om\subset \IR^n,\;n\ge 2,$ be a bounded domain, $\overline{\Om}$ its closure and $\p\Om$ its boundary. Let $f:\Om\times \IR\rightarrow \IR$ be continuous, $S^n$ denote the set of $n\times n$ symmetric matrices and
$H(p,X)$ be continuous, for $(p,X)\in\IR^n\times S^n$. We study properties of viscosity solutions of problems of the type
\eqRef{sec1.1}
H(Du, D^2u)+f(x,u)=0,\;\;\mbox{in $\Om$, and $u=h$ on $\p\Om$,}
\ee 
where $h\in C(\p\Om)$. By a solution we mean a function $u\in C(\overline{\Om})$ that solves (\ref{sec1.1}) in the viscosity sense. 

We require that the operator $H$ satisfy monotonicity in $X$, homogeneity in $p$ and $X$, a kind of coercivity and, in some instances, invariance under reflections and rotations, see below. Our work mainly studies questions related to the eigenvalue problem in (\ref{sec1.1}) although some of our work applies to a more general class of functions $f$. 
 
We now discuss the precise nature of $H$ and also state the main results of this work. Let $o$ denote the origin in $\IR^n$ and 
a point $x\in \IR^n$ will be occasionally written as $(x_1,x_2,\cdots, x_n)$. By $I$ we denote the $n\times n$ identity matrix, $O$ will denote the $n\times n$ matrix with all entries equalling zero. Also, $e$ will always stand for a unit vector in $\IR^n$.

Through out this work we require that $H\in C(\IR^n\times S^n, \IR)$ and $H(p, O)=0,\;\forall p\in \IR^n$. We now describe the conditions that $H$ satisfies.

{\bf Condition A (Monotonicity):} The operator $H(p,X)$ is continuous at $p=0$ for any $X\in S^n$ and $H(p, O)=0$, for any $p\in \IR^n$. In addition, for any $X,\;Y\in S^n$ with $X\le Y$,
\eqRef{sec2.4}
H(p,X)\le H(p,Y),\;\;\forall p\in \IR^n.
\ee
It is clear that if $X\ge O$ then $H(p, X)\ge 0$, for any $p$.

{\bf Condition B (Homogeneity):} There are constants $k_1\ge 0$ and $k_2>0$, an odd integer, 
such that for any $(p,X)\in \IR^n\times S(n)$,
\bea\label{sec2.41}
H(\tht p, X)=|\tht|^{k_1}H(p, X),\;\;\forall \tht\in \IR, \;\;\;\mbox{and}\;\;\;H(p, \tht X)=\tht^{k_2} H(p, X),\;\;\forall \tht>0.
\eea
Although our work allows $k_2\ge 1$, our interest is in the case $k_1=1$. Define
\eqRef{sec2.42}
k=k_1+k_2\;\;\;\mbox{and}\;\;\;\g=k_1+2k_2.
\ee

For the next condition, we work with the matrix $e\otimes e$, where $e\in \IR^n$ is a unit vector. 
Observe that $(e\otimes e)_{ij}=e_ie_j$ and $e\otimes e$ is a non-negative definite matrix.
 
{\bf Condition C (Coercivity):} $H$ is coercive in the following sense. Let $e$ denote a unit vector in $\IR^n$. For every $-\infty<s<\infty$, 
we set
\bea\label{sec2.430}
&&m_1(s)=\min_{|e|=1}H\left(e,I-s e\otimes e\right), \;\;\;m_2(s)=\max_{|e|=1}H\left(e,I-s e\otimes e\right),
\\
&&m_3(s)=\min_{|e|=1}H(e, s e\otimes e-I)\;\;\;\mbox{and}\;\;\;m_4(s)=\max_{|e|=1}H(e, s e\otimes e- I). \nonumber
\eea
If $H$ is odd in $X$ then $m_3(s)=-m_2(s)$ and $m_4(s)=-m_1(s)$. Also,
the functions $m_1(s)$ and $m_2(s)$ are decreasing in $s$ while $m_3(s)$ and $m_4(s)$ are increasing in $s$. If $s<1$ then $I-se\times e$ is a positive definite matrix and, by Condition A,
$m_1(s)\ge 0$ and $m_4(s)\le 0$.

Let $k_1,\;k_2$ and $k=k_1+k_2$ be as in (\ref{sec2.41}) and (\ref{sec2.42}). Set $\hat{s}=k_1/k$. We impose that
\eqRef{sec2.43}
m_1(\hat{s})>0\;\;\mbox{and}\;\;m_4(\hat{s})<0.
\ee 
More generally, we require that there are $-\infty<s_1\le 1\le s_0<\infty$ (see (\ref{sec2.4})) such that
\bea\label{sec2.44}
&&\mbox{(i)}\;\;\min\{m_1(s),\;-m_4(s)\}>0,\;\;\forall s\le s_1,\;\;\mbox{and}\\
&&\mbox{(ii)}\;\;\max\left\{ m_2(s), \;-m_3(s)\right\}<-\ell,\;\;\forall\;s\ge s_0, \nonumber
\eea
where $0<\ell<\infty$. 

With (\ref{sec2.430}), (\ref{sec2.43}) and (\ref{sec2.44}) in view, we set
\bea\label{sec2.461}
m_1(s)=\min\left\{ m_1(s), \;-m_4(s)\right\}\;\;\mbox{and}\;\;m_2(s)=\max\left\{m_2(s), -m_3(s)\right\}.
\eea
Also, with (\ref{sec2.42}) and (\ref{sec2.43}) in mind, we set
\bea\label{sec2.45}
m_1=m_1(\hat{s}),\;\;m_2=m_2(\hat{s}),\;\;\al=\frac{k_1+2k_2}{k}=\frac{\g}{k}\;\;\mbox{and} \;\;
\s=\frac{1}{\al m_1^{1/k}}.
\eea
In Part II of work, we will distinguish between the following two cases. 
\bea\label{sec2.450}
 \mbox{(i) $\exists$ $1<\bar{s}<2$ such that $m_2(\bar{s})<0$, or}\;\;\mbox{(ii) $\exists$ $\bar{s}\ge 2$ such that $m_2(s)<0,\;\forall s>\bar{s}.$}
 \eea

{\bf Condition D (Symmetry):} $H$ is invariant under rotations and reflections. 
As a result if $v(x)=v(r)$, where $r=|x-z|$ for some $z\in \IR^n$, then
\eqRef{sec2.46}
H(Dv, D^2v)=G(r, v^{\prime}(r), v^{\prime\prime}(r)).
\ee
Restated, $H(e, I-se\otimes e)$ is independent of $e$.

In Section 3, we discuss some examples and also make some additional comments. 
\vsp
We now address the eigenvalue problem. We take $a\in C(\Om)\cap L^{\infty}(\Om),\;\inf_\Om a>0$. Consider the problem of finding $(\lam, u)$ where $\lam\in \IR$ and $u\in C(\overline{\Om})$ solve
\eqRef{sec2.47}
H(Du, D^2u)+\lam a(x) |u|^{k-1}u=0,\;\;\mbox{in $\Om,$ and $u=h$ on $\p\Om$,}
\ee
where $h\in C(\p\Om)$ and $h>0$. If $(\lam, u)$ solves (\ref{sec2.47}) with $h=0$ then we say $\lam$ is an eigenvalue of the operator $H$ and $u\not\equiv 0$ an eigenfunction corresponding to $\lam$.
Our main effort is to characterize first eigenvalue and the first eigenfunction, see \cite{BNV,BL, J}. To this end, we study (\ref{sec2.47}) and show the existence of positive solutions when $h>0$ and when $\lam$ is less than a certain value $\lamo>0$ which turns out to be the first eigenvalue of $H$.

It turns out that when (\ref{sec2.450})(i) holds, conditions $A,\;B$ and $C$ suffice. However, (\ref{sec2.450})(ii) appears to be less tractable and we impose additional conditions. At this time
it is not clear to us as to how to prove a Harnack's inequality for non-negative super-solutions for such a general class of operators.

We now discuss the work in \cite{BD} that addresses the eigenvalue problem for nonlinear elliptic equations. Besides homogeneity they require that the operator $H$ satisfy $\forall (p, Y)\in \IR^n\times S^n,$
$$a |p|^qTrace(X)\le H(p, Y+X)-H(p, Y)\le b |p|^q Trace(X),\;\;\mbox{$\forall X\in S^n,\;X\ge 0$,} $$
where $0<a\le b<\infty$ and $q>-1$. 
Clearly, this condition implies that
\bea\label{sec2.60}
&&(i)\;\;a(t-s) \le H(e, I+t e\otimes e)-H(e, I+s e\otimes e)\le b (t-s),\;t\ge s,\;\;\mbox{and}  \nonumber\\
&&(ii)\;a\le \frac{H(e,I-e\otimes e)}{n-1}\le b.
\eea
Our conditions require that $H(p, X+Y)\ge H(p,X)$, for $Y\ge 0$, and coercivity as stated in condition C. Thus, $H(e, I-se\otimes e)$ is continuous and non-increasing in $s$ (see condition A), and that  (\ref{sec2.43}) and (\ref{sec2.44}) hold. The last two conditions are also satisfied by the operators in \cite{BD}. However, we do not require that $H$ be Lipschitz continuous, see (\ref{sec2.60})(i). Also, unlike (\ref{sec2.60})(ii), we allow
the possibility that $H(e, I-e\otimes e)=0$ (as in the case of the infinity-Laplacian). While \cite{BD} does not require condition D, the bounds in terms of the Laplacian support radial solutions. These being unavailable, we require that $H$ admit radial solutions  if (\ref{sec2.450})(ii) holds. Also, we require that $q\ge 0$ while the work in \cite{BD} allows $q>-1$. 

We also remark that our approach is different from \cite{BD}. We work the equation $H(Du, D^2u)+\lam a(x) u^k=0$ with positive boundary data, while in \cite {BD}, the authors work with the non-homogeneous equation $H(Du, D^2u)+\lam a(x)u^k=f(x)$ with zero boundary data. It is not clear if a version of Theorem \ref{sup1} and some of the estimates in Section 7 hold in their case. 
\vsp
We now state the main results of this work. The set $\Om\subset\IR^n,\;n\ge 2$, will always stand for a bounded domain in this work.
By $usc(\Om)$ we denote the class of all upper semi-continuous functions on $\Om$, and $lsc(\Om)$ will denote the class of all lower semi-continuous functions on $\Om$.  

The first result is a quotient type comparison principle for positive solutions, also see \cite{BD}.
Let $g,\;h:\Om\times \IR\rightarrow \IR$ and $a:\Om\rightarrow \IR,\;a>0,$ be continuous. Suppose that $m>0$ is such that 
\eqRef{sec4.2}
h(x, t)\ge a(x)|t|^{m-1}t>g(x,t),\;\forall (x,t)\in \Om\times (0,\infty)
\ee

\begin{thm}\label{sec4.3} Let $H$ satisfy conditions $A$ and $B$, and $g$ and $h$ be as in (\ref{sec4.2}). Suppose that $u\in usc(\Om)$, and $v\in lsc(\Om), \;v>0,$ solve 
$$H(Du, D^2u)+g(x,u)\ge 0,\;\;\mbox{and}\;\;H(Dv, D^2v)+h(x,v)\le 0,\;\;\mbox{in $\Om$.} $$
Recall $k$ from (\ref{sec2.42}). Then either $u\le 0$ in $\Om$, or the following conclusions hold.\\
\NI (a) Suppose that $k=m$.  
(i) If $U\subset \Om$ is a compactly contained sub-domain of $\Om$ such that $u>0$ somewhere in $U$ then
$$\sup_{U} \frac{u}{v}=\sup_{\p U} \frac{u}{v}>0.$$

 (ii) Assume that $u>0$ somewhere in $\Om$, and $U_j\subset U_{j+1}\subset \Om,\;j=1,2,\cdots$, are compactly contained sub domains of $\Om$, with $\cup_j U_j=\Om$. If 
$\lim_{j\rightarrow \infty} \sup_{U_j}u/v<\infty,$ then 
$$0<\sup_{\Om}\frac{u}{v}=\lim_{j\rightarrow \infty} \left(\sup_{U_j}\frac{u}{v}\right).$$
\NI(b) Take $k\ne m$. We assume further that either (i) $k>m$ and $(u/v)(z)>1$, for some $z$ in $\Om$, or (ii) $k< m$ and $\sup_{\Om}u/v<1$. Then the conclusions in (i) and (ii) of part (a) hold.
\end{thm}
\NI A related version of the comparison principle for a somewhat general case is discussed in Section 5.
See \cite{BD}.

For the remaining results, we assume that $H$ satisfies conditions $A,\;B$ and $C$. 

We now state a result on a priori supremum bounds that is useful for the eigenvalue problem for $H$ on $\Om$.
Let $f\in C(\Om\times \IR, \IR)$ and
\eqRef{sec2.48}
\sup_{\Om\times[t_1, t_2]} |f(x,t|<\infty,\;\;\;\mbox{$\forall t_1,t_2$ such that $-\infty<t_1\le t_2<\infty$.}
\ee 

Assume that
there are constants $-\infty<\mu_1\le 0\le \mu_2<\infty$ such that
\eqRef{sec5.12}
\limsup_{t\rightarrow \infty} \left( \frac{\sup_{\Om}f(x,t)}{t^{k}} \right)\le \mu_2\;\;\;\mbox{and}\;\;\;\liminf_{t\rightarrow -\infty} \left( \frac{\inf_{\Om}f(x,t)}{|t|^{k}} \right)\ge \mu_1.
\ee
\begin{thm}\label{sec5.13} Let $\lam\in \IR$, $f:\Om\times\IR\rightarrow \IR$ be as in (\ref{sec2.48}) and (\ref{sec5.12}), and $h\in C(\p\Om)$. Suppose that $u\in C(\overline{\Om})$ solves
$$H(Du, D^2u)+\lam f(x, u)=0,\;\;\mbox{in $\Om,$ and $u=h$ on $\p\Om$.}$$

(a) If $|\lam|$ is small enough then $u$ is a priori bounded and 
$\sup_{\Om} |u|\le K$,
where $K$ depends on $\lam, \mu_1, \mu_2, k, h$ and $\Om$. 

(b) If $\mu_1=\mu_2=0$ then, for any $\lam$, $u$ is a priori bounded and $\sup_\Om|u|\le K$, where $K$ depends on $\lam, k, h$ and $\Om$. 
\end{thm}

Let $\dl>0$ and $\lam>0$. Consider the problem of finding a positive solution $u_\lam\in C(\overline{\Om})$ of
\eqRef{sec2.49}
H(Du_\lam, D^2u_\lam)+\lam a(x) u_\lam^k=0,\;\;\mbox{in $\Om$, and $u_\lam=\dl$ on $\p\Om$.}
\ee
We call
\eqRef{sec2.50}
\lamo=\sup\{\lam:\;(\ref{sec2.49}) \;\mbox{has a positive solution $u_\lam$.} \}.
\ee
We show in Theorem \ref{sec7.9} that $\lamo>0$. 

The next result shows that a positive solution $u_\lam$, for any $0<\lam<\lamo$, is an increasing Lipschitz continuous function of $\lam$. 

\begin{thm}\label{sup1} Let $\lam>0,\;\dl>0$ and $u_\lam\in C(\overline{\Om}),\;u_\lam>0,$ solve
\eqRef{sec10}
H(Du_\lam, D^2u_\lam)+\lam a(x) u_\lam^k=0,\;\;\mbox{in $\Om$ and $u_\lam=\dl$ on $\p\Om$.}
\ee
Set $v_x(\lam)=u_\lam(x),\;\forall x\in \overline{\Om}$ and $M_\lam=\sup_\Om u_\lam$. Then for each $x\in \Om$, $v_x(\lam)$ is a non-decreasing Lipschitz continuous function of $\lam$ and for a.e. $\lam$,
$$\frac{v_x(\lam) \log (v_x(\lam)/\dl)}{k\lam} \le   \frac{dv_x(\lam)}{d\lam}\le \left(\frac{M_\lam}{k\dl}\right) \frac{v_x(\lam)-\dl}{\lam},\;\;\;0<\lam<\lamo.$$
\end{thm}
\NI Also, see Remark \ref{sup2}).

We now provide existence results for the eigenvalue problem in (\ref{sec2.47}). The conditions (\ref{sec2.450}) (i) and (ii) play a crucial role in these statements and the following will be assumed through out this work.
\bea\label{sec1.10}
&&\mbox{$\Om$ is any bounded domain if (\ref{sec2.450})(i) holds, and $\Om$ satisfies a uniform outer ball }\nonumber\\
&&\mbox{condition if (\ref{sec2.450}) (ii) holds.}
\eea
\begin{thm}\label{sec7.9}  Suppose that $\lam>0$, and $a(x)\in C(\Om)$ with $\inf_\Om a>0$. For $h\in C(\p\Om)$ with $\inf_{\p\Om} h>0$, consider the boundary value problem
\eqRef{sec7.10} 
H(Du, D^2u)+\lam a(x) |u|^{k-1} u=0,\;\;\mbox{in $\Om$ and $u=h$ on $\p\Om$.}
\ee
Set $R=$diam$(\Om)$ and $\nu=\sup_{\Om} a(x)$. Recall (\ref{sec1.10}).

(a) If (\ref{sec2.450})(i) holds and $0<\lam<|m_2(\bar{s})|(2-\bar{s})^k(\nu R^{\g})^{-1}$ then (\ref{sec7.10}) has a unique positive solution. 

(b) Suppose that (\ref{sec2.450})(ii) holds and $\Om$ satisfies a uniform outer ball condition with optimal radius $2\rho>0$. Fix $\beta>\bar{s}-2$ and $s=\beta+2$. If 
$$0<\lam<\frac{|m_2(s)|\beta^k}{\nu R^{\g}} \left(\frac{\rho}{R}\right)^{k\beta}$$ 
then (\ref{sec7.10}) has a unique positive solution. 
Moreover, $u>\inf_{\p\Om} h$ in $\Om$.     
\end{thm}

We show next  that $\lamo$, as defined in (\ref{sec2.50}), is independent of $h$. Also, see Remark \ref{sec8.80}.

\begin{thm}\label{sec7.171} Suppose that (\ref{sec1.10}) holds.
Let $\dl>0$, $a(x)\in C(\Om),\;\inf_\Om a>0,$ and $h\in C(\p\Om)$, with $\inf_{\p\Om} h>0$. Suppose that, for some $\lam>0$, the problem
$H(Du, D^2u)+\lam a(x) |u|^{k-1} u=0,\;\mbox{in $\Om,$ and $u=\dl$ on $\p\Om$,}$
has a positive solution. Then the problem
\eqRef{sec7.172} 
H(Dv, D^2v)+\lam a(x) |v|^{k-1} v=0,\;\;\mbox{in $\Om$ and $v=h$ on $\p\Om$,}
\ee
also has a positive solution.
\end{thm}

The boundedness of $\lamo$ is shown in

\begin{thm}\label{sec8.10} Suppose that $H$ satisfies conditions $A,\;B$ and $C$.
Let $\dl>0$ and $a(x)\in C(\Om)\cap L^{\infty}(\Om),\;\inf_\Om a>0$. Recall (\ref{sec1.10}).

(a) Suppose that (\ref{sec2.450}) (i) holds then $\lamo<\infty$. (b) Suppose that (\ref{sec2.450}) (ii) holds and $H$ satisfies $D$ then $\lamo<\infty$.
\end{thm}
In part (ii), if a Harnack's inequality holds then the conclusion follows without the imposition of condition D. See Remark \ref{sec8.14}.

Finally, we show

\begin{thm}\label{sec9.00} Suppose that $H$ satisfies conditions $A,\;B$ and $C$. Let $a\in C(\Om, \IR),\;\inf_\Om a>0.$ 
Consider the problem
\eqRef{sec9.001} 
H(Du, D^2u)+\lamo a(x) |u|^{k-1}u=0,\;\;\mbox{in $\Om$ and $u=0$ on $\p\Om$,}
\ee
where $u\in C(\overline{\Om})$. Recall (\ref{sec1.10}).

(i) Suppose that (\ref{sec2.450}) (i) holds then (\ref{sec9.001}) has a positive eigenfunction $u$.

(ii) Suppose that $H$ satisfies $D$ and (\ref{sec2.450}) (ii) holds then (\ref{sec9.001}) has a positive radial eigenfunction when $\Om$ is a ball. 
\end{thm}

At this time, it is not clear to us as to how to extend part (ii) to general domains. Also, our work does not address whether $\lamo$ is simple or isolated. 

We describe the lay out of the work. In Section 2, we include definitions, notations and some useful calculations. Section 3 contains examples and some further discussion.
Section 4 presents comparison principles when $H$ satisfies condition $A$ and are of some what general nature. The remaining work is divided into two parts. Part I has Sections 5 and 6. Sections 7-10 are in Part II. Section 5 lists additional comparison principles under the conditions $A$ and $B$ and the proof of Theorem \ref{sec4.3}. We also include a change of variables formula, important for Theorem \ref{sec8.10}. Sections 6-10 we assume that $H$ satisfies $A,\;B$ and $C$.
The proofs of Theorems \ref{sec5.10} and \ref{sec5.13} are in Section 6. Section 7 contains a discussion of questions related to the problem (\ref{sec2.49}) and shows that solutions $u_\lam$ are Lipschitz continuous in $\lam$. Proofs of Theorems \ref{sec7.9} and \ref{sec7.171} are in Section 8. Theorem \ref{sec8.10} is proven in Section 9. 
We present a proof of Theorem \ref{sec9.00} in Section 10. 

\section{Additional notations, definitions and calculations }

We introduce additional notations and definitions.
We use
$o$ to denote the origin. By $B_s(p),\;s>0$, we mean the ball of radius $s$ centered at $p$. In this work, all differential equations and inequalities will be understood in the sense of viscosity, see below and \cite{CIL}. We assume through out that $H\in C(\IR^n\times S^n, \IR)$ and satisfies condition $A$, see (\ref{sec2.4}). 

We define the notion of a viscosity solution $u$ to the following in $\Om$,
\eqRef{sec2.2}
H(Du, D^2u)+f(x,u)=0\;\;\;\mbox{in $\Om$ and $u=h$ on $\p\Om$},
\ee
where $f\in C(\Om\times \IR, \IR)$ and $h\in C(\p\Om)$. 

A function $u\in usc(\Om)$ is said to be a viscosity sub-solution of the equation in (\ref{sec2.2}), in $\Om$, or solves $H(Du, D^2u)+f(x,u)\ge 0$, in $\Om$, if the following holds. For any $\psi\in C^2(\Om)$ such that  $u-\psi$ has a maximum at a point $y\in \Om$, we have
$$H(D\psi(y), D^2\psi(y))+f(y, u(y))\ge 0.$$
Similarly, $u\in lsc(\Om)$ is said to be a viscosity super-solution of the equation (\ref{sec2.2}) or solves $H(Du, D^2u)+f(x,u)\le 0$, in $\Om$, if, for any $\psi\in C^2(\Om)$ such that $u-\psi$ has a minimum at $y\in \Om$, we have
$$H(D\psi(y),D^2\psi(y))+f(y, u(y))\le 0.$$
A function $u\in C(\Om)$ is a viscosity solution of if it is both a sub-solution and a super-solution. 

We define $u\in usc(\overline{\Om})$ to be a viscosity sub-solution to the problem (\ref{sec2.2})
if $u$ is a sub-solution in $\Om$ and $u\le h$ on $\p\Om$. Similarly, $u\in lsc(\overline{\Om})$ is a super-solution of (\ref{sec2.2}) if $u$ is a super-solution in $\Om$
and 
$u\ge h$ on $\p\Om$. We define $u\in C(\overline{\Om})$ to be a solution to (\ref{sec2.2}), if it is both a sub-solution and a super-solution of (\ref{sec2.2}).

In this work, we will utilize radial sub-solutions and super-solutions. We discuss (\ref{sec2.450}) in this context. Let $v(x)=v(r)$ where $r=|x-z|$, for some $z\in \IR^n$. Set $e=(e_1,e_2,\cdots,n)$ where $e_i=(x-z)_i/r,\;i=1,2,\cdots, n$. Then for $x\ne z$, 
\eqRef{sec2.5}
H(Dv, D^2v)=H\left(v^{\prime}(r)e, \frac{v^{\prime}(r)}{r}\left(I-e\otimes e\right)+v^{\prime\prime}(r) e\otimes e \right),
\ee
where $I$ is the $n\times n$ identity matrix. 
We now impose conditions $A,\;B$ and $C$ on $H$. Recall (\ref{sec2.41})-(\ref{sec2.450}), 
$$k=k_1+k_2,\;\;\g=k_1+2k_2\;\;\;\mbox{and}\;\; \al=\g/k.$$ 
Then (\ref{sec2.5}) reads
\eqRef{sec2.51}
H(Dv, D^2v)=\frac{|v^\prime(r)|^k}{r^{k_2}}H\left[ e, \pm \left\{ I- \left(1-\frac{rv^{\prime\prime}(r)}{v^\prime(r)}\right) e\otimes e\right\} \right].
\ee

Take $v(r)=c\pm d r^{\beta}$, where $d>0$ and $\beta>0.$ Using (\ref{sec2.51}). We obtain, in $r>0$,
\bea\label{sec2.6}
H(Dv, D^2 v)&=&H\left[ \pm d \beta r^{\beta-1} e, \pm d \beta r^{\beta-2} \{ I-(2-\beta) e\otimes e\} \right] 
\nonumber\\
&=&d^k \beta^k r^{(\beta-1)k_1+(\beta-2)k_2} H\left[e, \pm\left\{I-(2-\beta)e\otimes e\right\}\right]\nonumber\\
&=& (d\beta)^k r^{k\beta-\g}H\left[e, \pm \left\{I-(2-\beta)e\otimes e\right\}\right].
\eea
If $v=c\pm d r^{-\beta}$, where $d>0$ and $\beta>0,$ then (\ref{sec2.6}) yields
\eqRef{sec2.601}
H(Dv, D^2 v)=\frac{(d\beta)^k}{ r^{k\beta+\g}}H\left[e, \mp \left\{I-(\beta+2)e\otimes e\right\}\right].
\ee
If (\ref{sec2.450})(i) holds and $v^{\pm}=c\pm d r^\beta$ with $\beta=2-\bar{s}>0$ ($0<\beta<1$) then (\ref{sec2.461}) and (\ref{sec2.6}) imply
\bea\label{sec2.61}
&&H(Dv^+, D^2 v^+)=(d\beta)^k r^{k\beta-\g}H\left(e, I-\bar{s}e\otimes e\right)\le  -\frac{(d\beta)^k|m_2(\bar{s})|}{ r^{\g-k\beta}}<0,\;\;\mbox{and}\\
&&H(Dv^-, D^2 v^-)= (d\beta)^k r^{k\beta-\g}H\left(e,  \bar{s}e\otimes e-I\right)\ge \frac{(d\beta)^k|m_2(\bar{s})|}{ r^{\g-k\beta}}>0. \nonumber
\eea
If (\ref{sec2.450})(ii) holds and $v^{\pm}=c\pm dr^{-\beta}$ with $\beta>\bar{s}-2>0$, then (\ref{sec2.461}) and (\ref{sec2.601}) lead to
\bea\label{sec2.62}
&&H(Dv^+, D^2 v^+)=\frac{(d\beta)^k}{ r^{k\beta+\g}}H\left(e, \bar{s}e\otimes e-I\right)\ge  \frac{(d\beta)^k|m_2(s)|}{ r^{k\beta+\g}}>0, \;\;\mbox{and}\\
&&H(Dv^-, D^2 v^-)= \frac{(d\beta)^k}{ r^{k\beta+\g}}H\left(e, I-\bar{s}e\otimes e\right)   \le  -\frac{(d\beta)^k|m_2(s)|}{ r^{k\beta+\g}}<0, \nonumber
\eea
where $s=\beta+2$. As a second application, set $\beta=\al=\g/k$ in (\ref{sec2.6}) (see (\ref{sec2.45})) and take $v^{\pm}=c\pm dr^\al,\;d>0,$ to obtain
\bea\label{sec2.7}
&&H(Dv^+, D^2v^+)=(d\al)^{k} H\left(e, I-\frac{k_1}{k}e\otimes e \right)\ge (d\al)^k m_1=\left(\frac{d}{\s}\right)^k>0,\;\;\mbox{and}\nonumber\\
&&H(Dv^-, D^2v^-)=(d\al)^{k} H\left(e, \frac{k_1}{k}e\otimes e-I \right)\le -(d\al)^k m_1=-\left(\frac{d}{\s}\right)^k<0.
\eea

\begin{rem}\label{sec2.8} Our results on existence use Perron's method, see \cite{CIL}. The idea is as follows. Consider the problem of showing the existence of a solution $u\in C(\overline{\Om})$ of 
$$H(Du, D^2u)+f(x,u)=0,\;\;\mbox{in $\Om$ and $u=h$ on $\p\Om$,}$$
where $h\in C(\p\Om)$. Assume that the above admits a comparison principle. Let $\vep>0$, be a given small number. For each $y\in \p\Om$, we construct
(i) a sub-solution $v$ such that $v(y)=h(y)-\vep$ and $v\le h$ on $\p\Om$, and (ii) a super-solution $w$ such that
$w(y)=h(y)+\vep$ and $w\ge h$ on $\p\Om$. This implies existence of a solution $u$. $\Box$
\end{rem}

\section{Examples and further comments}

We discuss some examples to which our results apply. Let $s\in \IR$ and the vector $e\in \IR^n$ be such that $|e|=1$. Set $r=|x|,\;\forall x\in \IR^n$.

{\bf (i) $p$-Laplacian and the pseudo $p$-Laplacian:} The $p$-Laplacian $\D_p,\;p\ge 2,$ can be written as 
$D_p u=|Du|^{p-2}\D u+(p-2)|Du|^{p-4} \D_\infty u$, where $\D_\infty u=\sum_{i,j=1}^n D_iuD_ju D_{ij}u$ is the infinity-Laplacian. It is easy to see that
$H(e, I-se\otimes e)=(n+p-2)-(p-1)s.$

We consider a more general version i.e., $H(Du, D^2u)=|Du|^{q}\D u+a|Du|^{q-2} \D_\infty u,\;q\ge 0,$ then
$H(e, I-se\otimes e)=n-s+a(1-s),\;\forall s.$ Thus, 
$$\min(1,\;1+a)\le \frac{H(e, I-s e\otimes e)}{n-s}\le \max(1, 1+a),\;\;\forall a>-1\;\mbox{and}\;\forall s\le 1.$$

If $a>-1$ then conditions A, B, C and D, in Section 1, are satisfied. Hence, all the results of this work hold. 

Next we discuss a version of the pseudo $p$-Laplacian, which we denote by $\D_{p}^s$, where 
$$\D_{p}^su=|Du|^q\sum_{i=1}^n |D_iu|^p D_{ii}u,\;\;p,\;q\ge 0,\;\;\;\mbox{and} \;\;H=\D_p^s.$$
We observe, using Holder's inequality and that $|e_i|\le 1$, that $\sum_{i=1}^n |e_i|^{p+2}\le \sum_{i=1}^n|e_i|^p,$
\ben
\min\left(1,\; n^{(2-p)/2}\right)\le \sum_{i=1}^n |e_i|^p\le n,\;\;\mbox{and}\;\;\sum_{i=1}^n|e_i|^p\le \left( \sum_{i=1}^n |e_i|^{p+2}\right)^{p/(p+2)} n^{2/(p+2)}.
\een
Clearly, $H(e, I-se\otimes e)=\sum_{i=1}^n|e_i|^p- s\sum_{i=1}^n |e_i|^{p+2}$. Using the above, if $s>0$ then
\ben
\frac{1-s}{n^{|2-p|/2}}\le H(e, I-se\otimes e)&\le& \left(\sum_{i=1}^n|e_i|^{p+2}\right)^{p/(p+2)} 
\left[ n^{2/(p+2)} -s \left(\sum_{i=1}^n |e_i|^{p+2} \right)^{2/(p+2)}\ \right]\\
&\le & \left[ n^{2/(p+2)} -\frac{s}{n^{p/(p+2)}} \right] = \frac{n-s}{n^{p/(p+2)}}.
\een
If $e_i=1$, for some $i$, then $H(e, I-se\otimes e)=1-s$. Also, if $e_i=k^{-1/2}$, for $i=1,2,3,\cdots, k,$ and $e_i=0,\;i=k+1,\cdots, n,$ then $H(e, I-se\otimes e)=k^{-p/2}(k-s),\;k=1,\cdots, n$.

Conditions A, B and C hold and hence, Theorems 1.1-1.5 follow.
The operator does not have radial symmetry. However, a proof of the Harnack inequality and local Lipschitz regularity 
may be found in \cite{BK} (see Theorem 2.9 and Remark 2.11), also see \cite{B}. Clearly, Theorem 1.6 holds, see Remark \ref{sec8.14}. Moreover, a version of Lemma \ref{sec9.01} holds. Also, $H(Dr, D^2r)=r^{-1}H(e, I-e\otimes e)\ge 0$, the proof in Subsection I (take $\beta=1$ in (b)) shows that
Theorem 1.7 also holds.

{\bf (ii) $\infty$-Laplacian and a related operator:} Set $H(Du, D^2u)=\D_\infty u=\sum_{i,j=1}^nD_iu D_juD_{ij}u$. Thus,
$$H(e, I-se\otimes e)=\sum_{i,j=1}^n e_i e_j\left( \delta_{ij}-se_ie_j\right)=\sum_{i=1}^n e_i^2-s \left( \sum_{i=1}^n e_i^2\right)^2=1-s. $$
Clearly, all the conditions are met and all the results stated in Section 1 hold. 

For the following, we consider $q\ge 0$ and define
$H(Du, D^2u)=\sum_{i,j=1}^n|D_iu|^q|D_ju|^q D_i u D_j uD_{ij}u.$
Then
\ben
H(e, I-se\otimes e)=\sum_{i=1}^n |e_i|^{2q+2}-s \left(\sum_{i=1}^n |e_i|^{q+2} \right)^2.
\een
If $s\le 0$ then $H\ge 0$. Taking $s\ge 0$, writing $q+2=(q+1)+1$ and observing that $(\sum_{i=1}^n |e_i|^{q+2})^2\le \sum_{i=1}^n |e_i|^{2q+2}$, we get
$$H(e, I-se\otimes e)\ge (1-s) \sum_{i=1}^n |e_i|^{2q+2}\ge\frac{1-s}{n^{q}}.$$ 
Next, setting $\tht=(2q+2)/(q+2)$ and using $\sum_{i=1}^n |e_i|^{2q+2}\le \left(\sum_{i=1}^n |e_i|^{q+2} \right)^{\tht}$, we get
\ben
H(e, I-se\otimes e)&\le&
 \left(\sum_{i=1}^n |e_i|^{q+2} \right)^{\tht}\left(1-s\left(\sum_{i=1}^n |e_i|^{q+2} \right)^{2/(q+2)}\right)\\
 &\le &1-\frac{s}{ n^{q/(q+2)}}
 \een
Thus, Theorems 1.1-1.5 hold. The operator has no radial symmetry.
A Harnack's inequality and local Lipschitz continuity may be worked out 
along the lines of Theorem 2.9 and Remark 2.11 in \cite {BK}. Clearly, Theorem 1.6 holds.
Observing that $H(Dr, D^2r)=r^{-1}H(e, I-e\otimes e)\ge 0$, one can use the proof in Subsection I (take $\beta=1$ in (b)) to show that
Theorem 1.7 also holds.

{\bf (iii) Pucci operators:} Let $a_i$ denote an eigenvalue of the matrix $D^2u$. For $0<\lam\le\Lambda$ and $q\ge 0$ define  
\ben
M^{+,q}_{\lam, \Lam}(u)=|Du|^q\left( \Lam\sum_{a_i\ge 0} a_i+\lam \sum_{a_i\le 0}a_i\right)\;\;\mbox{and}\;\;M^{-,q}_{\lam, \Lam}(u)=|Du|^q\left( \lam\sum_{a_i\ge 0} a_i+\Lam \sum_{a_i\le 0}a_i\right).
\een
For any $e$, the eigenvalues of $I-se\otimes e$ are $1$, with multiplicity $n-1$, and $1-s$. Set $H^{\pm}(Du, D^2u)=M^{\pm, q}_{\lam, \Lam}(u)$ and observe that $H^+(e, \pm(I-se\otimes e))=- H^-(e, \mp(I-se\otimes e)).$
If $s\le 1$, then
\ben
H^+(e, I-se\otimes e)=\Lam(n-s)\;\;\;\mbox{and}\;\;\;H^+(e, s e\otimes s -I)=-\lam (n-s).
\een
If $s>1$, then
\ben
H^+(e, I-se\otimes e)=\Lam(n-1)+\lam(1-s)\;\;\mbox{and}\;\;\;H^+(e, s e\otimes s -I)=-\lam (n-1)-\Lam(1-s).
\een
The operators $H^{\pm}$ are radially symmetric.
The conditions A, B, C and D are satisfied and Theorems 1.1-1.5, 1.6(b) and 1.7(ii) hold. 

We can also consider the maximal and minimal Pucci operators \cite{GT}(Chap 17). Let
$$J=\{M(x)\in S^n:\;\sum_{i,j=1}^n M_{ij}(x)\eta_i\eta_j\ge a|\eta|^2\;\mbox{and}\;\sum_{i=1}^n M_{ii}(x)=1\}\;\;\mbox{and}\;\;0<a\le 1/n.$$
We set $H^{+(-)}(Du, D^2u)=\sup(\inf)_{M(x)\in J} |Du|^{k_1}Trace(M(x) D^2u),\;x\in \Om.$  
Then $H^+(p, X)=-H^-(p,-X).$ Set $E=a(n-s)$ and $F=(1-na)$. Then for any $e$ (see \cite{GT}), 
\ben
&&H^+(e, I-se\otimes e)=E+F,\;s\ge 0,\;\;\mbox{and}\;\;H^+(e, I-se\otimes e)=E+F(1-s),\;s\le 0,\\
&&H^-(e, I-se\otimes e)=E+F(1-s),\;\;s\ge 0,\;\mbox{and}\;\;H^-(e, I-se\otimes e)=E+F, \;s\le 0.
\een
Conditions A, B, C and D are satisfied and Theorems 1.1-1.5, 1.6(b) and 1.7(ii) hold. See \cite{QS} in this context.
\vsp
\section{Comparison principles under Condition A}

We assume $H$ satisfies condition A (see (\ref{sec2.4})) i.e., $H(p,X)$ is continuous on $\IR^n\times S^n$, $H(p,O)=0$ and, for any $X,\;Y\in S^n$ with $X\le Y,$ we have $H(p,X)\le H(p,Y),$ $\forall p\in \IR^n.$ The section begins with a version of a comparison principle that is used often in this work, also see \cite{BL}. 

\begin{thm}{(Comparison Principle)}\label{sec3.1} Let $f:\Om\times \IR\rightarrow \IR$ and $g:\Om\times \IR\rightarrow \IR$ be continuous. 
Suppose that $u\in usc(\overline{\Om})$ and $v\in lsc(\overline{\Om})$ satisfy in the viscosity sense,
$$H(Du, D^2u)+f(x,u(x))\ge 0\;\;\;\mbox{and}\;\;\;H(Dv, D^2v)+g(x,v(x))\le 0,\;\;\;\mbox{in $\Om.$}$$
If $\sup_{\Om}(u-v)>\sup_{\p\Om}(u-v)$ then there is a point $z\in \Om$ such that
$$(u-v)(z)=\sup_{\Om}(u-v)\;\;\;\mbox{and}\;\;\;\;g(z,v(z))\le f(z, u(z)).$$
Equivalently, if $\inf_{\Om}(v-u)<\inf_{\p\Om}(v-u)$ then there is a point $z\in \Om$ such that
$(v-u)(z)=\inf_{\Om}(v-u)\;\mbox{and}\;g(z,u(z))\le f(z,v(z)).$
\end{thm}

\begin{proof} We provide an outline, see \cite{CIL}. We prove part (a), the proof of part (b) follows similarly. 
Set $M=\sup _{\Om}(u-v)$; define, for $\vep>0$,
\be
w_{\vep}(x,y):=u(x)-v(y)-\frac{1}{2\vep}|x-y|^2,\;\;\;\forall(x,y)\in \Om\times \Om.
\ee
Set $M_{\vep}:=\sup_{\Om\times \Om}w_{\vep}(x,y)$, and let $(x_{\vep},\;y_{\vep})\in \overline{\Om}\times\overline{\Om}$ be such that $M_{\vep}$ is attained at $(x_{\vep},\;y_{\vep}).$

There is a $z\in \overline{\Om}$ such that $x_{\vep}$ and $y_{\vep}\rightarrow z,$ as $\vep\rightarrow 0$, and $M=(u-v)(z).$ Since $(x_{\vep},y_{\vep})$ is a point of maximum of $w_{\vep}(x,y)$, there exist
$X_{\vep}$ and $Y_{\vep}$ such that $((x_{\vep}-y_{\vep})/\vep, X_{\vep})\in \bar{J}^{2,+}u(x_{\vep})$ and $((x_{\vep}-y_{\vep})/\vep, Y_{\vep})\in \bar{J}^{2,-}v(y_{\vep})$. Moreover, we have $X_{\vep}\le Y_{\vep}$, and using the definitions of $\bar{J}^{2,+}$ and $\bar{J}^{2,-}$, we see that
\eqRef{sec3.3}
-f(x_{\vep},u(x_{\vep}))\le H((x_\vep-y_\vep)/\vep, X_\vep)  \le H((x_\vep-y_\vep)/\vep, Y_\vep) \le -g(y_{\vep}, v(y_{\vep})).
\ee
Now let $\vep\rightarrow 0$ to conclude that $g(z,v(z))\le f(z,u(z)).$  \end{proof}

As a consequence of Theorem \ref{sec3.1}, we state a version of the strong maximum principle that holds under some restrictions on $f$ and the boundary data. See \cite{BDA} for a more general version.

\begin{lem}\label{sec3.7}{(Maximum Principle)} Let $f\in C(\Om\times \IR, \IR)$. Let $J=\{c\in \IR\;:\;f(x, c)=0,\;\mbox{for some $x\in \Om$}\}.$
Call $c_{\inf}=\inf J$, $c_{\sup}= \sup J$ and assume that $-\infty<c_{\inf}\le c_{\sup}<\infty$. 

(i) Suppose that $f\le 0$ and $u\in usc(\overline{\Om})$ solves 
$H(Du, D^2u)+f(x,u)\ge 0,\;\mbox{in $\Om$}.$\\
If $\sup_{\p\Om} u> c_{\sup}$ or $\sup_{\Om} u<c_{\inf}$ then 
$$u(x)<\sup_{\p\Om}u,\;\forall x\in \Om, \;\;\;\mbox{and}\;\;\;\sup_{\overline{\Om}} u=\sup_{\p\Om} u.$$

(ii) Suppose that $f\ge 0$ and $u\in lsc(\overline{\Om})$ solves $H(Du, D^2u)+f(x,u) \le 0,\;\mbox{in $\Om$}.$
If $\inf_{\p\Om}u< c_{\inf}$ or $\inf_{\Om} u>c_{\sup}$ then
$$u(x)>\inf_{\p\Om} u,\;\forall x\in \Om,\;\;\;\mbox{and}\;\;\;\inf_{\overline{\Om}} u=\inf_{\p\Om} u$$
If in part (i) $f<0$, and in part (ii) $f>0$ then the corresponding conclusions hold without the stated restrictions on $u$ in $\p\Om$ and $\Om$.
\end{lem}
\begin{proof} We prove (i), the proof of (ii) is similar. Suppose that the claim is false i.e., there is a point $z\in\Om$ such that 
$u(z)=\sup_\Om u\ge \sup_{\p\Om} u$. By our hypothesis, $u(z)\notin J$ i.e., $f(z, u(z))\not= 0$. For $\vep>0$, small, define $\psi_\vep(x)=u(z)+\vep |x-z|^2$ in $\Om$. Then $\psi_\vep\in C^2(\Om)$, $(u-\psi_\vep)(z)= 0$ and 
$(u-\psi_\vep)(x)\le -\vep |x-z|^2<0,\;\forall x\in \overline{\Om},\;x\ne z.$
Thus, for any $\vep>0$, $z$ is the only point of maximum of $u-\psi_\vep$ in $\Om$. Using the definition of a viscosity sub-solution, we have 
$$H(D\psi_\vep(z), D^2\psi_\vep(z))+f(z, u(z))=H(0, 2 \vep I)+f(z, u(z))\ge 0.$$
Letting $\vep\rightarrow 0$, we get $0=H(0,O)\ge -f(z, u(z))\ge 0$. Thus, the claim holds.
\end{proof}

Finally, 

\begin{lem}\label{sec3.10} Let $g:\Om\times\IR\rightarrow \IR$ and $h:\Om\times \IR\rightarrow \IR$ be continuous. Suppose that (i) $g(x,t)<h(x,t),\;\forall(x,t)\in \Om\times \IR$ and
at least one of $g(\cdot,t)$ and $h(\cdot,t)$ is non-increasing in $t$, or (ii)  $g=h$ and $g$ is strictly decreasing in $t$. 

Let $u\in usc(\overline{\Om})$ and $v\in lsc(\overline{\Om})$ satisfy 
$H(D u, D^2u)+g(x,u)\ge 0$ and $H(D v, D^2v)+ h(x,v)\le 0$, in $\Om$. If $u\le v$, on $\p\Om$, then $u\le v$ in $\Om$. Moreover,
$$\sup_\Om (u-v)^+=\sup_{\p\Om}(u-v)^+.$$
\end{lem}
\begin{proof}
Suppose that $g$ is non-increasing in $t$ and $\sup_\Om(u-v)>\sup_{\p\Om}(u-v).$ By Theorem \ref{sec3.1}, there is a point $z\in \Om$ such that 
$(u-v)(z)=\sup_\Om(u-v)>0$ and $g(z,u(z)) \ge h(z, v(z))$. Since $u(z)>v(z)$, we get $g(z, v(z))\ge g(z,u(z))\ge h(z, v(z))$, a contradiction. Thus, the claim holds.

Next, set $\mu=\sup_{\p\Om}(u-v)$ and assume $\mu>0$. Define $u_\mu=u-\mu$ and observe that 
$$H(Du_\mu, D^2u_\mu)=H(Du, D^2u)\ge -g(x, u)\ge -g(x, u_\mu),\;\;\mbox{and}\;\;\sup_{\p\Om} (u_\mu-v)=0.$$
If $\sup_{\Om}(u_\mu-v)>0$ then, by Theorem \ref{sec3.1}, there is a point $z\in \Om$ such that $(u_\mu-v)(p)>0$ and
$g(z, u_\mu(z))\ge h(z, v(z))$. Since, $v(z)<u_\mu(z)$, $g(z, v(z))\ge g(z, u_\mu(z))\ge h(z, v(z)),$ a contradiction. Thus the claim holds.
\end{proof}

\centerline{PART I}

\section{Proof of Theorem \ref{sec4.3} and a change of variables result under conditions $A$ and $B$}

In this section, $H$ satisfies conditions A (monotonicity) and B (homogeneity), see (\ref{sec2.4}) and (\ref{sec2.41}).  We show some additional comparison principles including Theorem \ref{sec4.3}, see \cite{BL} and \cite{BD}. We also discuss a change of variables result.  Recall that $k=k_1+k_2$ in (\ref{sec2.42}).

{\bf Proof of Theorem \ref{sec4.3}:} Let $U\subset \Om$ be a compactly contained sub-domain of $\Om$. We assume that $u>0$ somewhere in $U$ and show that $\sup_{U}(u/v)=\sup_{\p U} u/v.$

Let $p\in U$ is such that $u(p)/v(p)=\sup_{U}(u/v)>\sup_{\p U} u/v.$ Set $\tau=u(p)/v(p)$. Since $v>0$, we have $u(p)>0$ and $\tau>0$. Thus, the function
\eqRef{sec4.4}
w(x)=u(x)-\tau v(x)\le 0,\;\;\;x\in \overline{U}.
\ee
Using (\ref{sec2.4}), (\ref{sec2.41}), (\ref{sec2.42}), (\ref{sec4.2}) and $v>0$, we have that
\bea\label{sec4.5}
 H(D\tau v , D^2\tau v)\le -\tau^kh(x,v(x))\le-\tau^{k-m} a(x)(\tau v(x))^m\;\;\;\forall x\in \Om.
\eea
From (\ref{sec4.4}), $w(x)<0$, for any $x\in \p U$, and $\sup_U w=w(p)=0$. Thus, $\sup_U w>\sup_{\p U}w$, and applying 
Theorem \ref{sec3.1} and (\ref{sec4.2}) to $u$ and $\tau v$ (see (\ref{sec4.5}))
there is a $z\in U$ such that 
\eqRef{sec4.6}
\mbox{$w(z)=\sup_U w=0$ and}\;\; g(z, u(z))\ge\tau^k h(z, v(z))\ge  \tau^{k-m}a(z) (\tau v(z))^m>0.
\ee
From (\ref{sec4.6}) we get $u(z)=\tau v(z)>0$ and
\eqRef{sec4.61}
g(z,u(z))\ge \tau^{k-m}a(z) (\tau v(z))^m= \tau^{k-m} a(z) u(z)^m> \tau^{k-m}g(z, u(z))>0.
\ee
We get a contradiction for $k=m$ and part (a)(i) holds. We prove part (a) (ii). Set $\mu_j=\sup_{\p U_j} (u/v)$. By part(i), $\mu_j$'s are increasing and $\mu=\sup_j\mu_j<\infty.$ If $\sup_{\Om}(u/v)>\mu$ then there is a set $U_j$ such that $\sup_{U_j}(u/v)>\mu$. This violates part (a)(i) since $\sup_{\p U_j}(u/v)\le \mu.$ 

From (\ref{sec4.61}), $\tau^{k-m}<1.$ We get a contradiction for part (b). Thus, the theorem holds. $\Box$

\begin{rem}\label{sec4.70}
Let $u,$ $v$, $g$ and $h$ be as in the statement of Theorem \ref{sec4.3}.  Since $v>0$, it follows that that if $u>0$ somewhere in $\Om$ then $u>0$ somewhere on $\p\Om.$
As a result, if $u=0$, on $\p\Om$, then $u\le 0$ in $\Om$. $\Box$
\end{rem}

\begin{rem}\label{sec4.700} The proof of Theorem \ref{sec4.3} can be adapted to the case 
$g(x,t)< a(x) |t|^{k-1} t+f(x)\le h(x,t),\;\forall (x,t)\in \Om\times (0,\infty)$, and where $f(x)\ge 0$. Suppose that
$v>0$ and $u\le v$ on $\p\Om$. If $u-v>0$ in $\Om$ then there is a $\tau>1$ such that
$\sup_\Om(u-\tau v)=0$. Clearly, $u-\tau v< 0$ on $\p\Om$. Arguing as in Theorem \ref{sec4.3} there is a point $p\in \Om$ such that $(u-\tau v)(p)=\sup_\Om(u-\tau v)=0,$ and 
$$g(p, u(p))\ge \tau^kh(p, v(p))\ge a(p) (\tau v(p))^k+\tau^k f(p)\ge a(p) u(p)^k+f(p)>g(p,u(p)).$$ 
This is a contradiction and $u\le v$ in $\Om$. Note the claim holds if we take $0<m\le k$.
A version may be found in \cite{BD} (Theorem 3.6). $\Box$
\end{rem}
 
Next, we extend Theorem \ref{sec4.3} when the condition (\ref{sec4.2}) is relaxed to include the case $g\le h$. 

\begin{lem}\label{sec4.71}
Let $a\in C(\Om),\;a>0,$ $u\in usc(\overline{\Om}),$ and $v\in lsc(\overline{\Om})\cap L^{\infty}(\Om),\;\inf_{\overline{\Om}} v>0.$ 
Assume that $g(x,t)\le a(x) |t|^{k-1} t\le h(x,t),\;\forall(x,t)\in \Om\times (0,\infty)$, where $k$ is as in (\ref{sec2.42}). 
If $u,\;v$ satisfy
$$H(Du, D^2u)+g(x,u) \ge 0,\;\;\;\mbox{and}\;\;\;H(Dv, D^2v)+h(x,v)\le 0,\;\; \mbox{in $\Om$},$$
and $u>0$ somewhere in $\Om$, then $\sup_{\Om}(u/v)= \sup_{\p\Om}(u/v).$
\end{lem}
\begin{proof} Since $\inf_\Om v>0$, we observe that $h(x,v(x))\ge a(x) v(x)^k>0,\;\forall x\in \Om$. Also, the proof of Lemma \ref{sec3.7} shows that $v> \inf_{\p\Om} v$, in $\Om$.

Set $\mu=\inf_{\p\Om}v$, $\ell=\sup_\Om v$, and $v_\tht=v-\tht\mu$, for $0<\tht<1$. Let $\vep>0$ be small, to be determined. Recalling that $h(x,t)\ge a(x) t^k,\;\forall t\ge 0$, we calculate
\bea\label{sec4.72}
H(Dv_\tht, D^2v_\tht)+(1+\vep)a(x) v_\tht^k&\le& h(x,v)\left(\frac{(1+\vep)v_\tht^k}{v^k}-1\right)\nonumber\\
&\le& \left( (1+\vep) \left( \frac{\ell-\tht\mu}{\ell} \right)^k-1 \right)h(x,v).
\eea
Note that in $\ell>0$, $(\ell-\tht\mu)/\ell$ increases in $\ell$. We choose $\vep$ small enough so that
$$H(Dv_\tht, D^2v_\tht)+(1+\vep)a(x) v_\tht^k\le 0,\;\;\mbox{in $\Om$.}$$ 
As done in the proof of Theorem \ref{sec4.3}, we take $\tau=\sup_\Om u/v$ and assume that $\tau>\sup_{\p\Om}u/v$. Working with $u,\;v_\tht$ and $w=u-\tau v_\tht$, we obtain that
there is a $z\in \Om$ such that $w(z)=\sup_U w=0$ and $g(z, u(z))\ge (1+\vep)a(z)(\tau v_\tht(z))^k>0.$
Since $\tau v_\tht(z)=u(z)$, 
$$g(z, u(z))\ge  (1+\vep) a(z) u(z)^k\ge (1+\vep) g(z, u(z))>0.$$
This is a contradiction and $\sup_\Om(u/v_\tht)=\sup_{\p\Om} (u/v_\tht).$ Letting $\tht\rightarrow 0$ proves the claim.
\end{proof}

We now present a result regarding a change of variables which will prove useful in Part II.

\begin{lem}\label{sec4.9}
Let $f\in C(\Om\times \IR,\IR)$ and $\zeta:\IR\rightarrow \IR$, be a $C^2$ function with $\zeta^{\prime}> 0$. Take
$u:\Om\rightarrow \IR$ and set $v=\zeta(u)$. Call $\eta=\zeta^{-1}$. 

(a) If $\zeta$ is convex and $u\in usc(\Om)$ satisfies $H(D u, D^2u)+f(x,u)\ge 0,\mbox{in $\Om$}$, then
$$H(Dv, D^2v)+[\eta^{\prime}(v(x))]^{-k}f(x, \eta(v(x)) \ge 0,\;\;\;\mbox{in $\Om$.}$$

(b) If $\zeta$ is concave and $u\in lsc(\Om)$ and $H(D u, D^2u)+f(x,u)\le 0,\mbox{in $\Om$}$, then
$$H(Dv, D^2v)+[\eta^{\prime}(v(x))]^{-k}f(x, \eta(v(x)) \le 0,\;\;\;\mbox{in $\Om$.}$$
\end{lem}
\begin{proof} We prove part (a). Clearly, $v\in usc(\Om)$. Since $\zeta$ is convex, we have
\eqRef{sec4.91}
\zeta(t_2)-\zeta(t_1)\ge \zeta^{'} (t_1) (t_2-t_1).
\ee
Let $\psi\in C^2(\Om)$ and $p\in \Om$ be such that $v(x)-\psi(x)$ has a maximum at $p$, i.e.,
$(v-\psi)(x)\le (v-\psi)(p).$ Thus, 
$\zeta^{'}(u(p)) (u(x)-u(p))\le \zeta(u(x))-\zeta(u(p))\le \psi(x)-\psi(p)$, see (\ref{sec4.91}).
Rearranging,
\eqRef{sec4.10}
u(x)-\frac{\psi(x)}{\zeta^{'}(u(p))}\le u(p)-\frac{\psi(p)}{\zeta^{'}(u(p))}. 
\ee 
Since $u$ is a sub-solution, we obtain
$$H \left( \frac{D\psi(p)}{ \zeta^{'}(u(p))},\; \frac{D^2\psi(p)}{\zeta^{'}(u(p))} \right)+ f(p, u(p))\ge 0.$$
Using $\zeta^{\prime}(\eta(t))\eta^{\prime}(t)=1$ together with (\ref{sec2.4}), (\ref{sec2.41}) and (\ref{sec2.42})  we rewrite the above as
$$H \left(D\psi(p), D^2\psi(p) \right)+[\eta^{\prime}(v(p))]^{-k} f(p, \eta(v(p)))\ge 0.$$
To prove part (b), we observe that the inequalities in (\ref{sec4.91}) and (\ref{sec4.10}) are reversed. One may now argue similarly to show part (b).
\end{proof}

\begin{rem}\label{sec4.12} Let $g$ and $h$, in $C(\Om,\IR)$, be such that
$g(x,t)\le a(x) t^k\le h(x,t),\forall(x,t)\in \Om\times \IR^+.$ Suppose that $u:\Om\rightarrow \IR^+$ and $\zeta(t)=t^\beta,\;t\ge 0$. Lemma \ref{sec4.9} implies the following.

(i) Let $u\in usc(\Om)$ solve $H(Du, D^2u)+g(x,u) \ge 0$, in $\Om$. If $\beta>1$ then
$H(Dv, D^2v)+\beta^k a(x) v^k\ge 0.$

(ii) Let $u\in lsc(\Om)$ solve $H(Du, D^2u)+h(x,u)\le 0$, in $\Om$.
If $\beta<1$ then 
$H(Dv, D^2v)+\beta^k a(x) v^k\le 0.$ $\Box$
\end{rem}

Next, Theorem \ref{sec4.3}, Lemma \ref{sec4.71} and Remark \ref{sec4.12} imply the following comparison principle.

\begin{lem}\label{sec4.120} Let $g$, $h$ be in $C(\Om, \IR)$ and 
$g(x,t)\le a(x) |t|^{k-1} t\le h(x,t),\;\forall(x,t)\in \Om\times (0,\infty)$, where $a\in C(\Om)$ and $a>0.$ 

Suppose that (i) $0<\beta<1$ is such that
$g(x,t)\le \beta^k a(x) t^k,\;\forall t\ge 0$, and (ii) $\sup_{x\in \Om} |h(x,t)|=0$ iff $t=0$. 
Let $u\in usc(\overline{\Om})$ and $v\in lsc(\overline{\Om})$, $\inf_{\overline{\Om}}v>0,$ solve 
$$H(Du, D^2u)+g(x,u)\ge 0,\;\;\mbox{and}\;\; H(Dv, D^2v)+h(x,v)\le 0,\;\;\mbox{in $\Om$.}$$
Assume that $u>0$ somewhere in $\Om$. If $u\le v$ on $\p\Om$ then $u\le v$ in $\Om$. Also,
\eqRef{sec4.121}
\sup_{\Om}\frac{u}{v^{\beta}}\le \sup_{\p\Om}\frac{u}{v^{\beta}}.
\ee
If $0<\sup_{\p\Om}u\le \inf_{\p\Om}v$, on $\p\Om$, then $u<v$, in $\Om$, and
$u(z)\le (\inf_{\p\Om} v)^{1-\beta}v(z)^\beta,\;\;\forall z\in \Om.$
\end{lem}
\begin{proof} By Theorem \ref{sec4.3} and Lemma \ref{sec4.71}, if $u\le v$ on $\p\Om$ then $u\le v$, in $\Om$. 
By Remark \ref{sec4.12} and the lower bound for $h$, $w=v^\beta$ solves
$H(Dw, D^2 w)+\beta^k  a(x) w^k\le 0.$ From Lemma \ref{sec4.71} and the upper bound for $g$ it is seen that (\ref{sec4.121}) holds.
By Lemma \ref{sec3.7}, $v>\inf_{\p\Om}v$. Next, using $u\le v$ in $\Om$, we get for any $z\in \Om$,
$$\frac{u(z)}{v(z)^\beta}\le \sup_{\p\Om} \left(\frac{u}{v^\beta}\right)\le \frac{\sup_{\p\Om} u}{\inf_{\p\Om} v^\beta}\le\frac{\inf_{\p\Om} v}{\inf_{\p\Om} v^\beta} \le \inf_{\p\Om} v^{1-\beta}<v(z)^{1-\beta}.$$
Thus, $u(z)<v(z)$. The above inequality also implies $ u(x)\le \inf_{\p\Om} v^{1-\beta}v^\beta(x), \;\forall x\in \Om.$  
\end{proof}

\section{A priori bounds: Proof of Theorem \ref{sec5.13}}

In this section, we derive some useful a priori bounds for a fairly general class of functions $f(x,t)$. We assume that $H$ satisfies conditions A, B and C, see (\ref{sec2.4})-(\ref{sec2.450}). However, we make no use of (\ref{sec2.450}) in this section.

We state the following version of the maximum principle, see Lemma \ref{sec3.7} in this context.

\begin{lem}\label{sec4.13}{(Maximum principle)} \\
(i) If $u\in usc(\overline{\Om})$ solves $H(Du, D^2u)\ge 0$, in $\Om$, then
$\sup_{\Om}u=\sup_{\p\Om} u$. 
(ii) If $u\in lsc(\overline{\Om})$ solves $H(Du, D^2u)\le 0$, in $\Om$, then
$\inf_{\Om}u=\inf_{\p\Om} u$. 
\end{lem}
\begin{proof} Let $q\in \IR^n\setminus \overline{\Om}$ and $0<\rho<R<\infty$ be such that $\Om \subset B_R(q)\setminus B_\rho(q).$
We prove (i) by contradiction (part (ii) is similar). Let $\vep>0$ and $p\in \Om$ be such that $u(p)\ge \sup_{\p\Om} u+\vep.$ Define
$$w(x)=\sup_{\p\Om} u+ \frac{\vep}{2} \left( \frac{R^2-|x-q|^2}{R^2-\rho^2}\right),\;\;\forall x\in B_R(q)\setminus B_\rho(q).$$
Thus, $\sup_{\p\Om} u\le w(x)\le \sup_{\p\Om}u+\vep/2$, in $B_R(q)\setminus B_\rho(q)$. Clearly, $u(p)-w(p)>0$ and
$u-w\le 0$, on $\p\Om$. Let $z\in \Om$ be a point of maximum of $u-w$ on $\overline{\Om}$. By (\ref{sec2.4}), (\ref{sec2.41}), (\ref{sec2.43}), (\ref{sec2.461}) and (\ref{sec2.45}),
$$H(Dw(z), D^2w(z))= \left( \frac{\vep}{R^2-\rho^2} \right)^k|x-q|^{k_1} H(e, -I)\le -m_1\rho^{k_1}\left( \frac{\vep}{R^2-\rho^2} \right)^k<0, $$
where $e$ is a unit vector in $\IR^n$ and $m_1\le m_1(0)$. We get a contradiction and the claim holds. 
\end{proof}

\begin{rem}\label{sec5.0} Let $u^{\pm}(x)=c\pm d|x|^\al$, $d>0$ and $\al=1+k_2/k$, see (\ref{sec2.43}), (\ref{sec2.461}) and (\ref{sec2.45}). 

We remark that (\ref{sec2.7}) holds in the viscosity sense at $r=0$. Suppose that $\psi\in C^2$ is such that $(u^+-\psi)(x)\le (u^+-\psi)(o).$ Then $d|x|^\al\le \langle D\psi(o), x\rangle+o(|x|)$ as $x\rightarrow 0$. If $x=-\vep D\psi(o),\;\vep>0,$ we get
$d\vep^{\al-1}\le -|D\psi(o)|^{2-\al}+o(1),$ as $\vep\rightarrow 0$. Thus, $D\psi(o)=0$ and
$d|x|^\al\le \langle D^2\psi(o)x, x\rangle/2+o(|x|^2)$, as $x\rightarrow o$. Since $1<\al<2$, this is a contradiction. The inequality $H(Du^+, D^2u^+)\ge (\al d)^km_1>0$ holds, see (\ref{sec2.7}). Next, let $\phi\in C^2$ be such that $(u^- -\phi)(x)\ge (u^- -\phi)(o)$. Then, $-d|x|^\al\ge \langle D\phi(o),x\rangle+o(|x|)$ and $D\phi(o)=0$. Thus, $-d|x|^\al\ge \langle D^2\phi(o)x, x\rangle/2+o(|x|^2)$, as $|x|\rightarrow 0$.
This is a contradiction and $H(Du^-, D^2u^-)\le -(\al d)^k m_1$, see (\ref{sec2.43}), (\ref{sec2.461}) and (\ref{sec2.7}). $\Box$
\end{rem}

We consider the problem
\eqRef{sec5.1}
H(Du, D^2u)+f(x, u)=0,\;\;\in \Om\;\;\mbox{and $u=h$ on $\p\Om$,}
\ee
where $h\in C(\p\Om)$ and $f\in C(\Om\times \IR, \IR)$. For a function $g$, define $g^+=\max\{g,0\}$ and $g^-=\min\{g,0\}$. We now present a priori supremum bounds when $f(x,u)=f(x)$, also see \cite{BMO2}.

\begin{lem}\label{sec5.7}
Let $f\in C(\Om)\cap L^\infty(\Om)$, $h\in C(\p\Om)$, $\al$ and $\s$ be as in (\ref{sec2.45}).
Suppose that $B_{R_o}(z_o)$, for some $z_0\in \IR^n,$ is the out-ball of $\Om$.
Consider the problem
\eqRef{sec5.8}
H(Du,D^2 u)+f(x)=0,\;\;\forall x\in\Om,\;\;\;u= h,\;\;\mbox{on}\;\;\p\Om,
\ee

(i) If $u\in usc(\overline{\Om})$ is a sub-solution of (\ref{sec5.7}) then 
$\sup_\Om u\le   \sup_{\partial \Om}h+\s (\sup_{\Om} f^+)^{1/k}R_o^{\al}.$

(ii) Similarly, if $u\in lsc(\overline{\Om})$ is a super-solution of (\ref{sec5.7}) then $\inf_\Om u\ge \inf_{\partial \Om}h-\s |\inf_{\Om} f^-|^{1/k}R_o^{\al}.$
\end{lem} 
\begin{proof}  We prove part (i). Let $u$ be a sub-solution of (\ref{sec5.8}). Fix $\vep>0$ and consider the function
$$
w_\vep (x)=\sup_{\partial\Om}h+\sigma(\sup_\Om f^+ +\vep)^{1/k}\left(R_o^{\al}-|x-z_0|^{\al}\right), \forall x\in \overline{\Om}.
$$
Applying (\ref{sec2.41}), (\ref{sec2.42}), (\ref{sec2.43}), (\ref{sec2.461}), (\ref{sec2.45}), (\ref{sec2.7})
and Remark \ref{sec5.0}, we have
$$H(Dw_\vep, D^2w_\vep)=\left( \frac{\sup_\Om f^+ +\vep}{m_1}\right)H\left(e, \frac{k_1}{k} e\otimes e-I \right)\le -\sup_{\Om} f^+  -\vep<-f,\;\;\mbox{in $\Om$}.$$
Also, $w_\vep\ge h$, on $\p\Om$. Thus, Lemma \ref{sec3.10} implies $u(x)\le w_\vep(x)$, in $\Om$. Since $\vep$ is arbitrary the claim follows. 

Part (ii) follows by taking
$\hat{w}_\vep(x)=\inf_{\Om}h-\s (|\inf_\Om f^-| +\vep)^{1/k} (R_o^\al -|x-z_0|^\al ),\;\;\forall x\in \overline{\Om}.$
\end{proof}
\vsp
Let $f:\Om\times \IR\rightarrow \IR$ be continuous and satisfy
\eqRef{sec5.9}
\sup_{\Om\times[t_1, t_2]} |f(x,t|<\infty,\;\;\;\mbox{$\forall t_1,t_2$ such that $-\infty<t_1\le t_2<\infty$.}
\ee 
We apply Lemma \ref{sec5.7} to prove Theorem \ref{sec5.13}. A related result is proven in \cite{BMO2}.

\NI{\bf Proof of Theorem \ref{sec5.13}.} Set $M=\sup_{\Om}| u|$, $L=\sup_{\p\Om} |h|$ and $R_o$ the radius of the out-ball of $\Om$. Let $\vep>0$, small, be fixed. 

We prove part (a). By (\ref{sec5.12}) there exists $t_1>0$ such that  
\eqRef{sec5.10}
(\mu_1-\vep)|t|^k\le\inf_\Om f(x,t)\le \sup_\Om f(x,t)\le (\mu_2+\vep)|t|^k,\;\;\;\forall |t|>t_1.
\ee
By (\ref{sec5.9}), there is a $0<\mu_3<\infty$ such that $\sup_{[-t_1,t_1]\times \Om} |f(x,t)|\le \mu_3$. Define $s_4=\max( |\mu_1|+\vep,\;\mu_2+\vep)$, then 
$\sup_{\Om}|f(x,t)|\le \mu_4|t|^k+\mu_3,\;\;-\infty<t<\infty.$

Thus, we have $|f(x,u)|\le \mu_4 M^k+\mu_3$ and
\eqRef{sec5.15}
-|\lam|(\mu_4 M^k+\mu_3)\le H(Du,D^2u)\le |\lam|(\mu_4 M^k+\mu_3),\;\;\mbox{in $\Om$.}
\ee
Using $-L\le u\le L$ on $\p\Om$ and
applying the estimates of Lemma \ref{sec5.7} to (\ref{sec5.15}) we have
$$0\le |u|\le M\le L+\s|\lam|^{1/k}(\mu_4M^{k}+\mu_3)^{1/k}R_o^{\al}\le L+|\lam|^{1/k} \mu_5 M+ \mu_6 .$$
where $\mu_5>0$ and $\mu_6>0$ are independent of $M$.
Hence, $M\le (L+\mu_6)(1-|\lam|^{1/k}\mu_5)^{-1}.$
It is clear that if $|\lam|$ is small enough $u$ is a priori bounded. 

To show (b), take $\vep=\left\{ |\lam|(2\s R_o^\al)^{k}\right\}^{-1}$, and we get from (\ref{sec5.10}), $|f(x,t)|\le \vep |t|^k,\;|t|\ge t_1>0,$ where $t_1>L.$ Suppose that $M>t_1$, then
$H(Du, D^2u)\le \vep\lam M^k$ in the set $\{u>t_1\}$. Using Lemma \ref{sec5.7},
$M\le t_1+\s (\vep \lam)^{1/k} M$. Applying the definition of $\vep$, $\sup_\Om u\le 2t_1$. A similar argument can be used to obtain a lower bound for $\inf_\Om u.$
$\Box$

\centerline{Part II}

\section{ Estimates for the eigenvalue problem. Proof of Theorem \ref{sup1}.} 
 
In Part II, $H$ satisfies conditions $A,\;B$ and $C$, see (\ref{sec2.4})-(\ref{sec2.45}).
From hereon, $\lam\in \IR$ stands for a parameter, $a\in C(\Om, \IR)$ and $\dl\ge 0$. Assume that there are $0<\mu\le \nu<\infty$ such that
\eqRef{sec7.4}
0<\mu\le a(x)\le \nu<\infty,\;\;\forall x\in \Om.
\ee
We study the problem
\eqRef{sec7.5}
H(Du, D^2u)+\lam a(x) |u|^{k-1}u=0,\;\;\mbox{in $\Om$ and $u=h$ on $\p\Om$,}
\ee
where $u\in C(\overline{\Om})$, $h\in C(\p\Om)$ and $\inf_{\p\Om}h>0$. 

We record an observation for (\ref{sec7.5}) and include a comment relevant to the eigenvalue problem for $H$ when $H$ is odd in $X$.

\begin{rem}\label{sec7.14}
(i) Let $\lam>0$, $a(x)$ be as in (\ref{sec7.4}) and $h\in C(\p\Om),\;\inf_\Om h> 0$. If $u\in C(\overline{\Om}),\; u>0,$ solves 
$$H(Du, D^2u)+\lam a(x) |u|^{k-1}u= 0,\;\;\mbox{in $\Om$ and $u=h$ on $\p\Om$,}$$
then $u>\inf_{\p\Om}h$ in $\Om$ and is unique. These follow from Lemmas \ref{sec3.7} and \ref{sec4.71}.
We will show in Lemma \ref{sec7.7} that for small $\lam>0$ any solution $u$ is necessarily positive.

(ii) Suppose that $H(p, -X)=-H(p, X),\;\forall(p,X)\in \IR^n\times S^{n}$. We show that if (\ref{sec7.5}) has a positive super-solution for some $\lam>0$ then any solution of (\ref{sec7.5}) is necessarily positive. Hence, if, for some $\lam>0$, a solution changes sign, then (\ref{sec7.5}) has 
no positive super-solutions.

Let $v\in C(\overline{\Om}),\;v>0,$ 
solve $H(Dv, D^2v)+\lam a(x) v^{k}\le 0,\;\;\mbox{in $\Om$, and $v\ge h$ on $\p\Om$.}$ Let $u$ be any solution of (\ref{sec7.5}) that changes sign in $\Om$. Set $\Om^-=\{u<0\}$. Take $w=-u$, then $w>0$, in $\Om^-$, $H(Dw, D^2w)+\lam a(x) w^k= 0$, in $\Om^-,$ and 
$w=0$, on $\p\Om^-.$ Use Lemma \ref{sec4.71} in every component of $\Om^-$ to conclude that $\sup_{\Om^-}w/v\le \sup_{\p\Om^-}(w/v)=0.$ Thus, $\Om^-=\emptyset$ and $u\ge 0$. Lemma \ref{sec3.7} yields that $v>\inf_{\p\Om} h$. Uniqueness follows from Lemma \ref{sec4.71}. $\Box$
\end{rem}

Next, recalling (\ref{sec7.4}) and applying the estimates of Lemma \ref{sec5.7} to a solution $u$ of (\ref{sec7.5}), we get
$$\inf_{\p\Om}h -\s |\nu \lam|^{1/k} R_o^{\al}|\inf_{\Om} u^{-}|\le u(x)\le \sup_{\p\Om}h+\s|\nu\lam|^{1/k}R_o^{\al} \sup_{\Om} u^+.$$
where $\al$ and $\s$ are as in (\ref{sec2.45}). Setting $\Lambda=\nu^{-1} (\s R_o^{\al})^{-k},$ we obtain
\eqRef{sec7.6}
\inf_{\p\Om}h -\left(\frac{|\lam|}{\Lambda}\right)^{1/k} |\inf_{\Om}u^{-}|\le u(x)\le \sup_{\p\Om}h+\left(\frac{|\lam|}{\Lambda}\right)^{1/k}\sup_{\Om} u^+,
\ee

Our next result discusses the influence of $\lam$ on the solutions of (\ref{sec7.5}). 
\begin{lem}\label{sec7.7}
Let $a\in C(\Om)$ be as in (\ref{sec7.4}) and $\Lambda$ be as in (\ref{sec7.6}). Suppose that $u\in C(\overline{\Om})$ solves
$$H(Du, D^2u)+\lam a(x) |u|^{k-1} u=0,\;\;\mbox{in $\Om$, and $u=h$ on $\p\Om$,}$$
where $h\in C(\p\Om)$. Set $\kappa_1=\inf_{\p\Om} h$ and $\kappa_2=\sup_{\p\Om}h$. Then the following hold.

(i) If $\lam\le 0$ then $\min(0,\kappa_1)\le u\le \max(0,\kappa_2)$ in $\Om$. If $\lam=0$ then 
$\kappa_1\le u\le \kappa_2$ in $\Om.$

(ii) If $h=0$ and $u$ is a non-zero solution then $\lam>0$. (iii) If $0<\lam<\Lambda$ then
 \ben
 \frac{\kappa_1}{1-(\lam/\Lambda)^{1/k} }\le \inf_\Om u^-\le u(x)\le\sup_\Om u^+\le \frac{\kappa_2}{1-(\lam/\Lambda)^{1/k} }.
 \een
In particular, if $h\ge 0$ then $\kappa_1\le u\le \tht \kappa_2$, where $\tht=(1-(\lam/\Lambda)^{1/k} )^{-1}.$ Thus, if $\lam>0$ is small and $h>0$ then any solution $u$ is positive in $\Om$.
 \end{lem}
 \begin{proof} We use Lemma \ref{sec4.13}. We prove part (i). Let $\lam\le 0$ and $\Om^-=\{x\in \Om:\;u(x)<\min (0,\kappa_1) \}$ be non-empty. Then $H(Du, D^2u)\le 0$, in $\Om^-$, and this contradicts 
 Lemma \ref{sec4.13}(ii).
Next, if $\Om^+=\{x\in \Om:\;u(x)>\max(0, \kappa_2)\}$ is non-empty then $H(Du, D^2u)\ge 0$ in $\Om^+$. This contradicts Lemma \ref{sec4.13} (i). Part (ii) now follows as a contrapositive of (i). To show (iii), we use (\ref{sec7.6}) and conclude that
$$\frac{\inf_{\p\Om} h}{1-(\lam/\Lambda)^{1/k} }\le \inf_\Om u^-\le u(x)\le\sup_\Om u^+\le \frac{\sup_{\p\Om} h}{1-(\lam/\Lambda)^{1/k} }.$$
If $\inf_{\p\Om} h\ge 0$ then $\inf_\Om u^-=0$ and we obtain the final estimate in the lemma.
 \end{proof}
 
We now show that if (\ref{sec7.5}) has a positive super-solution then it has a super-solution for a slightly larger value of $\lam$ and a sub-solution for a smaller value of $\lam$.
 
\begin{thm}\label{sec7.15}
Let $a\in C(\Om)$ be as in (\ref{sec7.4}), $h\in C(\p\Om),\;\inf_{\p\Om}h>0,$ and $\lam>0$.
Suppose that $u\in L^{\infty}(\Om)$ and $u>0$. Set $\vt=\inf_{\p\Om} h$; define $v_\tht=(u-\tht\vt)/(1-\tht),\;\forall 0\le \tht<1,$ and $w_\tht=(u+\tht\vt)/(1+\tht),\;\forall\tht>0$.

(i) If $u\in lsc(\overline{\Om})$ solves $H(Du, D^2u)+\lam a(x) u^k\le 0,\;\mbox{in $\Om,$}$ and 
$u\ge h$ on $\p\Om$, then, for every 
$$0<\vep\le \tht\lam k \left( \frac{(\vt/\sup_\Om u)}{1-\tht(\vt/\sup_{\Om} u)} \right),\;\;0\le \tht< 1,$$
the function $v_\tht$ solves
$ H(Dv_\tht, D^2v_\tht)+(\lam+\vep) a(x) v_\tht^k\le 0,\;\mbox{in $\Om,$ and $v_\tht\ge h$ on $\p\Om$.}$

(ii) Suppose that $u\in usc(\overline{\Om})$ solves $H(Du, D^2u)+\lam a(x) u^k\ge 0,\;\mbox{in $\Om,$}$ and 
$u\le h$ on $\p\Om$. For every $0<\vep<\lam$ there is a $\tht>0$ such that $w_\tht$ solves
$$ H(Dw_\tht, D^2w_\tht)+(\lam-\vep) a(x) w_\tht^k\ge 0,\;\;\mbox{in $\Om,$ and $w_\tht\le h$ on $\p\Om$.}$$
\end{thm}
\begin{proof} Set $m=\sup_\Om u$. 

 (i) Fix $0<\tht\le1$ and set $\eta=u-\tht\vt$. By Lemma \ref{sec4.13}, $u\ge \vt$, and  $\eta\ge (1-\tht)\vt$, in $\Om$, and $\eta\ge (1-\tht)h$ on $\p\Om$.
Observe that $(t-\tht\vt)/t,\;t\ge \vt,$ is increasing in $t$. Calculating,
\ben
H(D\eta, D^2\eta)+(\lam+\vep)a(x)\eta^k&\le& a(x) \left\{ (\lam+\vep)\eta^k-\lam u^k \right\}=a(x) u^k \left\{ (\lam+\vep) \left(\frac{u-\tht\vt}{u}\right)^k-\lam \right\}\nonumber\\
&\le &a(x)u^k\left\{ (\lam+\vep) \left(\frac{m-\tht\vt}{m}\right)^k-\lam \right\}\le 0,
\een
if we choose $0<\vep\le \lam \left\{ m^k(m-\tht\vt)^{-k}-1\right\}$. Using the lower bound $k(t-1)\le t^k-1,\;t\ge 1$, we take
$$0<\vep\le \tht \lam k \left( \frac{ (\vt/m)}{1-\tht(\vt/m)} \right).$$
Using the homogeneity of $H$, $v_\tht(x)= \eta/(1-\tht)=(u-\tht\vt)/(1-\tht),\;0\le \tht<1,$
$\forall x\in \Om,$ solves $H(Dv_\tht, D^2v_\tht)+(\lam+\vep) a(x) v_\tht^k\le 0$, in $\Om$, and $v_\tht\ge h$, on $\p\Om$. 

(ii) Let $0<\vep<\lam$ be fixed and $\tht>0$ to be determined. Set $\varphi=u+\tht\vt$, in $\overline{\Om}$ and calculate to obtain
\ben
H(D\varphi, D^2\varphi)+(\lam-\vep) a(x) \varphi^k &\ge& a(x) u^k \left( (\lam-\vep) \frac{\varphi^k}{u^k}-\lam \right)=a(x) u^k \left( (\lam-\vep) \left( \frac{u+\tht\vt}{u} \right)^k-\lam \right)\\
&\ge&a(x) u^k\left( (\lam-\vep) \left( \frac{m+\tht\vt}{m}\right)^k-\lam \right) \ge 0,
\een
if $\tht$ is such that $0<\vep\le \lam \left\{ (m+\tht\vt)^k m^{-k}-1\right\}.$ Clearly, $w_\tht\ge \vt$, in $\Om$, and $w_\tht\le h$ on $\p\Om$.
\end{proof}

We introduce a quantity that will be useful for the eigenvalue problem. 
Let $\dl>0$ and $\lam>0$. Consider the problem of finding a positive solution $u_\lam\in C(\overline{\Om})$ of
\eqRef{sec7.16}
H(Du_\lam, D^2u_\lam)+\lam a(x) |u_\lam|^{k-1}u_\lam=0,\;\;\mbox{in $\Om$, and $u_\lam=\dl$ on $\p\Om$.}
\ee
Define
\eqRef{sec7.1700}
\lamo=\sup\{\lam:\;(\ref{sec7.16}) \;\mbox{has a positive solution $u_\lam$.} \}.
\ee
We will show in Sections 8 and 9 that $0<\lamo<\infty.$ For the next result we assume this fact. 

\begin{thm}\label{sec7.17} 
Suppose that $\lamo>0$, where $\lamo$ is as in (\ref{sec7.1700}) and $0<\lam<\lamo$. Let $a\in C(\Om)$ be as in (\ref{sec7.4}) and $u_\lam>0$ be a solution of (\ref{sec7.16}). Then the following hold.

(i) The solution $u_\lam$ is unique and $u_\lam>\dl$. (ii) For every $x\in \Om$, the function $u_\lam(x)$ increases as $\lam$ increases.
(iii) Call $m_\lam=\sup _{\Om} u_\lam$. If $0<\lamo<\infty$ then 
$$m_\lam\ge \dl \left(1+\frac{k\lam}{\lamo-\lam}\right),\;\;0<\lam<\lamo.$$  
Thus, $m_\lam\rightarrow \infty$ as $\lam\rightarrow \lamo$.

(iv) The set of $\lam$'s for which (\ref{sec7.16}) has a positive solution is the interval $[0,\lamo)$.
\end{thm} 
\begin{proof} Parts (i) and (ii) follow from Remark \ref{sec7.14} and Theorem \ref{sec4.3}. 
We use Theorem \ref{sec7.15} (i) to prove part (iii). To see this, let $u_\lam$ be the solution of 
(\ref{sec7.16}) for some $\lam<\lamo$. Let $0<\tht<1$ and $\vep$ be as in part (i) of Theorem \ref{sec7.15}. Then
$$0<\vep\le \tht \lam k \left( \frac{ (\dl/m_\lam)}{1-\tht(\dl/m_\lam)} \right)\le \lam k \left( \frac{ (\dl/m_\lam)}{1-(\dl/m_\lam)} \right),\;\;\;\forall \;0<\tht<1.$$
Clearly, $\lam+\vep\le \lamo$. Letting $\tht\uparrow 1$,
$$\lamo-\lam\ge\lam k \left( \frac{ (\dl/m_\lam)}{1-(\dl/m_\lam)} \right).$$
Rearranging, we obtain the estimate in part (iii).

For part (iv), let $\lam<\lamo$ and (\ref{sec7.16}) have a solution $u_\lam$. If $0<\hat{\lam}<\lam$ then $H(Du_\lam, D^2u)+\hat{\lam} a(x) u_\lam(x)^k\le 0,$ with $u_\lam=\dl$ on $\p\Om$. Thus, $u_\lam$ is super-solution and $v=\dl$ is a sub-solution. By Lemma \ref{sec4.71} and Remark \ref{sec2.8}, the problem $H(Dw, D^2w)+\hat{\lam} a(x) w^k=0,$ in $\Om$, with $w=\dl$, has a positive solution.

From Theorem \ref{sec7.15}(i), $v_\tht$ is a super-solution of (\ref{sec7.16}) with $\lam+\vep$ and
$v_\tht=\dl$ on $\p\Om$. Also, the function $v=\dl$ is a sub-solution in $\Om$. By Lemma \ref{sec4.71} and Remark \ref{sec2.8} there is a positive solution of $H(Dw,D^2w)+(\lam+\vep)a(x) w^k=0$, in $\Om$, and $w=\dl$ on $\p\Om$. 
\end{proof} 
Theorem \ref{sec7.15} will be instrumental for proving Theorem \ref{sup1}. We now show that $u_\lam$, a solution of (\ref{sec7.16}), is an increasing Lipschitz continuous function of $\lam$, for $0<\lam<\lamo.$
\vsp
{\bf Proof of Theorem \ref{sup1}.}

\begin{proof} By Theorem \ref{sec7.17} (iv), $\lam$ is an interior point. Fix $\lam$ and $x\in \Om$.
Thus, (\ref{sec10}) has a unique positive solution $u_\lam$ and
by Theorem \ref{sec4.3}, $v_x(\lam)=u_\lam(x)$ is non-decreasing in $\lam$. Set $M_\lam=\sup u_\lam.$

We make repeated use of Theorem \ref{sec7.15} (i). Recall that if $u_\lam$ solves (\ref{sec10}) then for $0<\tht<1$ and
$0<\vep\le \tht k \lam (\dl/M_\lam)$ there is a solution $u_{\lam+\vep}>0$ of
\eqRef{sec10.1}
H(Du_{\lam+\vep}, D^2u_{\lam+\vep})+(\lam+\vep) a(x) u_{\lam+\vep}^k=0,\;\;\mbox{in $\Om$, and $u_{\lam+\vep}=\dl.$}
\ee
Also, $w(x)=(u_{\lam}-\tht\dl)(1-\tht)^{-1}$ is a super-solution of (\ref{sec10.1}), and by Lemma \ref{sec4.71},
\eqRef{sec10.2}
u_{\lam+\vep}\le w,\;\;\mbox{in $\Om$.}
\ee

{\bf Upper Bound} By (\ref{sec10.1}) and (\ref{sec10.2}), for $0<\tht<1$ and $0<\vep\le \tht k\lam(\dl/M_\lam),$ we have
$$v_x(\lam+\vep)=u_{\lam+\vep}(x)\le \frac{u_{\lam}(x)-\tht\dl}{1-\tht}=\frac{v_x(\lam)-\tht\dl}{1-\tht}.$$
Taking $\vep=k\tht\lam_0(\dl/M_\lam)$, $\tht$ small, we get
\ben
0\le \frac{v_x(\lam+\vep)-v_x(\lam)}{\vep}\le \frac{\tht(v_x(\lam)-\dl)}{(1-\tht)\vep}=\left(\frac{M_\lam}{k\dl}\right) \frac{v_x(\lam)-\dl}{\lam(1-\tht)}
\een
Thus, the right hand derivative $D^+_\lam v_x(\lam)\le (M_\lam/k\dl) (v_x(\lam)-\dl)/\lam,$ by letting $\tht\rightarrow 0$.

We now compute the left hand derivative $D^-_\lam M_{\lam}$ as follows. Fix $0<\tht<1$, small, and choose $\hat{\lam}<\lam$ such that 
(see (\ref{sec10.1}) and (\ref{sec10.2}))
\eqRef{sec10.3}
\hat{\lam}=\lam\left(1-k\tht\left(\frac{\dl}{M_{\hat{\lam}}}\right)\right)\;\;\;\mbox{and}\;\;\;\vep=k\tht \left(\frac{\dl \hat{\lam}}{M_{\hat{\lam}}}\right)
\ee
Thus, $\lam=\hat{\lam}+\vep .$ Observe that $v_x(\hat{\lam})\le v_x(\lam)$ and $M_{\hat{\lam}}\le M_\lam$.

Using (\ref{sec10.2}), (\ref{sec10.3}), Theorem \ref{sec7.17} and Lemma \ref{sec4.71} 
yield $v_x(\hat{\lam}+\vep)=v_x(\lam)\le (v_x(\hat{\lam})-\tht\dl)/(1-\tht),$ and
\ben
0\le \frac{v_x(\hat{\lam}+\vep)-v_x(\hat{\lam})}{\vep}=\frac{v_x(\lam)-v_x(\hat{\lam})}{\vep}\le \frac{\tht(v_x(\hat{\lam})-\tht\dl)}{\vep(1-\tht)} 
\le \left(\frac{M_\lam}{k\dl}\right) \frac{v_x(\lam)-\dl}{\hat{\lam}(1-\tht)}.
\een
Letting $\tht\rightarrow 0$, we get $D^-_\lam v_x(\lam)\le (M_\lam/k\dl) (v_x(\lam)-\dl)/\lam.$ Clearly, $v_x(\lam)$ is Lipschitz continuous for fixed $x$ and $\dl>0$. The upper bound in the theorem holds.

{\bf Lower Bound.} 
Let $0<\lam_1<\lam_2<\lamo$. Using (\ref{sec10}), Remark \ref{sec4.12} and Lemma \ref{sec4.120},
$$v_x(\lam_1)/\dl\le ( v_x(\lam_2)/\dl )^{\tau},\;\;\mbox{where}\;\;\tau=( \lam_1/\lam_2 )^{1/k}<1.$$ 
We obtain $\log (v_x(\lam_1)/\dl)\le \tau \log( v_x(\lam_2)/\dl)$
Subtracting $\tau\log (v_x(\lam)/\dl)$ from both sides, rearranging and noting that $v_x(\lam)$ is Lipschitz continuous, we see that
$$\frac{\log (v_x(\lam_2)/\dl)-\log (v_x(\lam_1)/\dl) }{\lam_2-\lam_1}\ge \left( \frac{\lam_2^{1/k}-\lam_1^{1/k} }{\lam_1^{1/k} (\lam_2-\lam_1) }\right) \log (v_x(\lam_1)/\dl).$$
The conclusion follows by letting $\lam_2\rightarrow \lam_1$.
\end{proof}

\begin{rem}\label{sup2} The above theorem holds if $v_x(\lam)$ is replaced by $M_\lam=\sup v_x(\lam)$. $\Box$ \end{rem}

\vsp
\section{Existence: Proofs of Theorems \ref{sec7.9} and \ref{sec7.171}}
 
We now present the proof of Theorem \ref{sec7.9} and show the existence of a solution of (\ref{sec7.5}), for small $\lam>0$ which in turn will imply that $\lamo>0$. This is done by constructing suitable sub-solutions and super-solutions. The two cases in (\ref{sec2.450}) are addressed separately. 

{\bf Proof of Theorem \ref{sec7.9}.} Set 
$$\nu=\sup_{x\in \Om} a(x),\;\;m=\inf_{\p\Om}h,\;\;\bar{m}=\sup_{\p\Om}h,\;\;R=\mbox{diam}(\Om),\;\;
\mbox{and assume that}\;\;m>0\;\;\mbox{and}\;\;\nu<\infty.$$ 
Using (\ref{sec2.61}) and (\ref{sec2.62}), we construct suitable sub-solutions and super-solutions to achieve our goal. 

Fix $y\in \p\Om$ and $\vep>0$, small, be such that $m-2\vep>0$. Let $r=|x-y|$. By continuity, there is a 
$\eta>0$ such that 
\eqRef{sec7.100}
h(y)-\vep\le h(x)\le h(y)+\vep,\;\;\forall x\in \overline{B}_\eta(y)\cap \p\Om.
\ee

{\bf Case (a):} Suppose that (\ref{sec2.450})(i) holds. Fix $\beta=2-\bar{s}$. Let $v^\pm=c\pm dr^\beta$ where $c>0$ and $d>0$. By (\ref{sec2.461}) and (\ref{sec2.61}), 
\bea\label{sec7.101}
H\left( Dv^+, D^2v^+ \right)\le  \frac{(d\beta)^k}{ r^{\g-k\beta}} m_2(\bar{s})<0,\;\;\mbox{and}\;\;\;H\left( Dv^-, D^2v^- \right)\ge  \frac{(d\beta)^k}{ r^{\g-k\beta}} |m_2(\bar{s})|>0.
\eea
Note that $0<\beta<1$, $k=k_1+k_2$, $\g=k_1+2k_2$ and $\g-k\beta>0$, see (\ref{sec2.45}). 

We construct a sub-solution $v$. We assume that $h(y)>m$, otherwise we take $v(x)=m$ in $\overline{\Om}.$ Note that $h(y)-\vep-(m/2)>0$. Set $r=|x-y|$ and
\bea\label{7.1011}
v^-(x)=\left\{ \begin{array}{ccc} h(y)-\vep- \left(h(y)-\vep-m/2\right)\left( r/\eta\right)^{\beta}, & \mbox{in $\overline{B}_\eta(y)\cap \overline{\Om}$},\\
m/2, &\mbox{in $\overline{\Om}\setminus \overline{B}_\eta(y).$} \end{array}\right.
\eea 
Then $v^-(y)=h(y)-\vep$, $v^-=m/2$, on $\p B_\eta(y)$, and
$m/2\le v^-(x)\le h(y)-\vep$, in $\Om$.

Applying (\ref{sec7.101}), we obtain that $H(Dv^-, D^2v^-)\ge 0$ and 
$v^-$ is a sub-solution of (\ref{sec7.10}) in $B_\eta(y)\cap \Om$. Also, $v^-$ is a sub-solution in $\Om\setminus \overline{B}_\eta(y)$ and by (\ref{sec7.100}) $v^-\le h$ on $\p\Om$. 
To show that
$v^-$ is a sub-solution in $\Om$, let $p\in \p B_\eta(y)\cap \Om$ and $\psi\in C^2$ be such that $(v^- -\psi)(x)\le (v^- -\psi)(p)$. Since $v^-(p)=m/2$ and $v^-(x)\ge v^-(p)$, we get
$0\le \langle D\psi(p), x-p\rangle+o(|x-p|)$ as $x\rightarrow p$. It follows that $D\psi(p)=0$ and a second order expansion shows that $D^2\psi(p)\ge 0$. Clearly, $H(D\psi(p), D^2\psi(p))+\lam a(p)(v^-(p))^k\ge 0$.

Next, we construct a super-solution $v^+$. We assume that $h(y)<\bar{m}$, otherwise take $v^+(x)=\bar{m}$ in $\Om$. For a fixed $0<\tht<1$, let $\lam=\tht |m_2(\bar{s})|(2-\bar{s})^k(R^{\g}\nu)^{-1}$. Set $r=|x-y|$ and
$$v^+(x)=h(y)+\vep+dr^\beta,\;\;\mbox{in $\overline{\Om}$,}$$
where $\beta=2-\bar{s}$ and
$$d\ge \max\left( \frac{ 2\bar{m}-h(y)-\vep }{\eta^\beta},\; \frac{ 2\bar{m}\tht^{1/k} }{(1-\tht^{1/k})R^\beta} \right).$$
It is easy to see that $v^+(y)=h(y)+\vep$, $v^+\ge 2\bar{m}$, on $\overline{\Om}\setminus B_\eta(y)$, and by (\ref{sec7.100}) $v^+\ge h$ on $\p\Om$. 
Set $c=h(y)+\vep$, and observing that $\g-k\beta>0$, we calculate, using the value of $\lam$,
\ben
&&H(Dv^+, D^2v^+)+\lam a(x) (v^+)^k\le \lam \nu (c+dR^\beta)^k-\frac{\beta^kd^k}{ R^{\g-k\beta}} |m_2(\bar{s})|\\
&&=(c+dR^\beta)^k \left( \lam \nu -\frac{\beta^k|m_2(\bar{s})|}{R^\g} \left( \frac{d R^\beta}{c+dR^\beta} \right)^k \right)
\le(c+dR^\beta)^k \left( \lam \nu -\frac{\beta^k|m_2(\bar{s})|}{R^\g} \left( \frac{d R^\beta}{2\bar{m}+dR^\beta} \right)^k \right)\\
&&\le(c+dR^\beta)^k \left( \lam \nu -\frac{|m_2(\bar{s})|\beta^k\tht}{R^\g}  \right)=0,
\een
where we have used $t/(1+t),\;t>0,$ is increasing in $t$ and $dR^\beta\ge  2\bar{m}\tht^{1/k} (1-\tht^{1/k})^{-1}.$ Thus, $v^+$ is a super-solution. Lemma \ref{sec4.71} and Remark \ref{sec2.8} imply existence of a solution $u$ of (\ref{sec7.10}).

{\bf Case (b):} Let (\ref{sec2.450}) (ii) hold i.e, $\Om$ satisfies a uniform outer ball condition. Call $2\rho$ as the optimal radius. Let $z\in \IR^n$ be such that 
$B_{2\rho}(z)\subset \IR^n\setminus \Om$ and $y\in \p B_{2\rho}(z)\cap \p\Om$. Choose $\eta<\rho$. 

Using (\ref{sec2.62}) fix $\beta>\bar{s}-2\ge 0$ and $s=\beta+2$. Taking $v^\mp=c\pm dr^{-\beta}$, 
\eqRef{sec7.102}
H\left( Dv^+, D^2v^+ \right)\ge  \frac{(d\beta)^k|m_2(s)|}{r^{k\beta+\g}} >0, \;\;\mbox{and}\;\;H\left( Dv^-, D^2 v^- \right)\le  - \frac{(d\beta)^k|m_2(s)|}{ r^{k\beta+\g}}<0.
\ee

We construct a sub-solution as follows. Let $p$, on the segment $\overline{yz}$, be such that $|y-p|=\eta/4$. Clearly, $\Om\cap \left(B_{\eta/2}(p)\setminus \overline{B}_{\eta/4}(p)\right)$ is 
a non-empty open set. Set $r=|x-p|$.

Assume that $h(y)>m$ and $m-2\vep>0$. We take $v^-(x)=c+dr^{-\beta}$ in $B_{\eta/2}(p)$, where 
$$c=h(y)-\vep-\frac{4^\beta d}{\eta^\beta},\;\;\mbox{and}\;\;d=\eta^\beta \left( \frac{h(y)-\vep-(m/2)}{4^\beta-2^\beta}\right)>0.$$
Thus, $v^-(y)=v^-(\eta/4)=h(y)-\vep$, $v^-(x)=m/2$, on $\p B_{\eta/2}(p),$ and 
$m/2\le v^-\le h(y)-\vep$. Extend $v^-=m/2$ in $\overline{\Om}\setminus B_{\eta/2}(p)$. 

By (\ref{sec7.102}),
$H(Dv^-, D^2v^-)+\lam a(x) (v^-)^k\ge 0,$ in $B_{\eta/2}(p)\cap \Om$. Since $B_{\eta/2}(p)\subset B_\eta(y)$, (\ref{sec7.100}) implies $v^-\le h$ on $\p\Om$. The proof that $v^-$ is a sub-solution in $\Om$ is similar to that in Case (a).  

We construct super solutions as follows. Take $0<\tht<1$ and 
\eqRef{sec7.103}
\lam=\tht \frac{|m_2(s)|\beta^k}{\nu R^{\g}} \left(\frac{\rho}{R}\right)^{k\beta}.
\ee
Recalling the outer ball condition, select $q$ on the segment $\overline{yz}$ such that $|q-y|=\rho$ and set $r=|x-q|$. There is a $\bar{\rho}>\rho$ such that $B_{\bar{\rho}}(q)\cap \Om\subset B_\eta(y)\cap \Om$ ($2\rho$ is the optimal radius).  Set    
$v^+(x)=c-dr^{-\beta}$, where
$$c=h(y)+\vep+\frac{d}{\rho^\beta},\;\;\mbox{and}\;\;\;d\ge \max \left\{ \frac{ (\rho\bar{\rho})^\beta (2\bar{m}-h(y)-\vep) }{\bar{\rho}^\beta-\rho^\beta},\; \frac{ 2\bar{m} \tht^{1/k} \rho^\beta}{1-\tht^{1/k}}\right\}.$$
Clearly, $v^+(y)=v^+(\rho)=h(y)+\vep$, $v^+\ge 2\bar{m}$, in $\overline{\Om}\setminus B_{\bar{\rho}}(p),$ and (\ref{sec7.100}) implies $v^+\ge h$ on $\p\Om$. Next, using (\ref{sec7.102}), (\ref{sec7.103}) and $d\rho^{-\beta}\ge2\bar{m} \tht^{1/k}(1-\tht^{1/k})^{-1}$, we calculate, in $\Om\setminus \overline{B}_\rho(q)$,
\ben
H(Dv^+, D^2v^+)+\lam a(x) (v^+)^k&\le& \lam \nu (c-dr^{-\beta})^k- \frac{(d\beta)^k|m_2(s)|}{R^{k\beta+\g}}\\
&\le &  \tht (2\bar{m}+d\rho^{-\beta})^k \frac{|m_2(s)|\beta^k}{R^{\g}} \left(\frac{\rho}{R}\right)^{k\beta}-\frac{\beta^kd^\beta |m_2(s)|}{R^{k\beta+\g}} \\
&\le &(2\bar{m}+d\rho^{-\beta})^k\frac{|m_2(s)|\beta^k}{R^{\g}} \left(\frac{\rho}{R}\right)^{k\beta} \left\{ \tht- \left(\frac{d\rho^{-\beta}}{2\bar{m}+d\rho^{-\beta}}\right)^k \right\}\le 0.
\een
Thus, $v^+$ is a super-solution and
Lemma \ref{sec4.71} and Remark \ref{sec2.8} imply existence. $\Box$

\begin{rem}\label{sec7.160} The proof of Theorem \ref{sec7.9} shows that, unlike the super-solutions, the constructions of the sub-solutions in Cases (a) and (b) are independent of $\lam$ for $\lam\ge 0.$ Also, the upper bounds for $\lam$ in the two cases do not depend on the boundary data $h$.  $\Box$
\end{rem}  

We show a domain monotonicity property of $\lamo$ (see (\ref{sec7.1700})) i.e,
if $\Om^{'}\subset \Om$ then $\lamo\le \lam_{\Om^{'}}.$ We use this in proving Theorem \ref{sec7.171}. This is shown for any subdomain if (\ref{sec2.450}) (i) holds and for any subdomain that satisfies a uniform outer ball condition if (\ref{sec2.450}) (ii) holds, see (\ref{sec1.10}). 

\begin{lem}\label{sec8.7}
Let $\Om^{'}\subset \Om$ be a sub-domain. Suppose that $\lam>0$, $\dl>0$ and $a\in C(\Om)\cap L^{\infty}(\Om),\;a>0.$ Assume that for some $0<\lam<\infty$ the problem
\eqRef{sec8.8}
H(Du, D^2u)+\lam a(x) |u|^{k-1}u=0,\;\mbox{in $\Om,$ and $u=\dl$ on $\p\Om$},
\ee
has a positive solution $u\in C(\overline{\Om})$. Then the problem 
$H(Dv, D^2v)+\lam a(x) |v|^{k-1}v=0,$ in $\Om^{'}$, and $v=\dl$ on $\p\Om^{'}$,
has a positive solution $v\in C(\overline{\Om^{'}}).$ 
\end{lem}
\begin{proof} By Theorem \ref{sec7.9} and (\ref{sec7.1700}), $\lam_{\Om^{'}}>0$. Assume that $\lambda_{\Om^{'}}<\infty$ otherwise the lemma holds.

Suppose that (\ref{sec8.8}) has a solution for $\lam\ge \lambda_{\Om^{'}}.$ By Theorem \ref{sec7.17} (iv) and (\ref{sec7.1700}), for any 
$0<\hat{\lam}<\lambda_{\Om^{'}}$, there is a solution $v_{\hat{\lam}}\in C(\overline{\Om^{'}}),\;v_{\hat{\lam}}>0,$ of
$$H(Dv_{\hat{\lam}}, D^2v_{\hat{\lam}})+\hat{\lam} a(x) v_{\hat{\lam}}^k=0,\;\mbox{in $\Om^{'}$, and $v_{\hat{\lam}}=\dl$ on $\p\Om^{'}$.}$$
By Theorem \ref{sec7.17} (i), $u\ge \dl$ on $\p\Om^{'}$. Applying Theorem \ref{sec4.3} in $\Om^{'}$, $v_{\hat{\lam}}\le u$ for every $\lam^{'}\in (0,\lambda_{\Om^{'}})$. Since $u$ is bounded, this contradicts Theorem \ref{sec7.17}(iii). The claim holds and $\lam< \lambda_{ \Om^{'}}.$
\end{proof}

\begin{rem}\label{sec8.80} Take $h\in C(\p\Om)$ with $\inf_{\p\Om}h>0$. For some $\lam>0$, let $u>0$ solve $H(Du, D^2u)+\lam a(x) u^k=0$, in 
$\Om$, and $u=h$ on $\p\Om$. Recall $\lamo$ from (\ref{sec7.1700}) and let $\lam\ge\lamo$. 
For every $t<\lamo$, let
$v_t>0$ solve $H(Dv_t,D^2v_t)+ta(x) v_t^k=0$, in $\Om,$ and $v_t=\inf_{\p\Om} h$ on $\p\Om$. By Theorem \ref{sec4.3}, $v_t\le u,\;\forall t$. This contradicts
Theorem \ref{sec7.17}(iii). Thus, $H(Dw, D^2w)+\lam a(x) w^k=0$, in $\Om$ and $w=\dl>0$, on $\p\Om$ has a solution. $\Box$
\end{rem}

We now prove Theorem \ref{sec7.171} and show that $\lamo$ in (\ref{sec7.1700}) is independent of the data $h\in C(\p\Om),\;h>0.$ See 
Remark \ref{sec7.160}.

{\bf Proof of Theorem \ref{sec7.171}.} Let $\lam>0$ such that there is a solution $u\in C(\overline{\Om}),\;u>0,$ of 
$H(Du, D^2u)+\lam a(x) u^k=0,\;\mbox{in $\Om,$ and $u=\dl$ on $\p\Om$}.$ We show that
\eqRef{sec7.1710}
H(Dv, D^2v)+\lam a(x) v^k=0,\;\mbox{in $\Om,$ and $v=h$ on $\p\Om$},
\ee
can be solved for any boundary data $h\in C(\p\Om)$ with $\inf_{\p\Om} h>0$.

By Remark \ref{sec7.160}, the sub-solutions in Theorem \ref{sec7.9} can be utilized here. Thus, our effort is to construct super-solutions to (\ref{sec7.1710}) for a given $h$.
Set
$$\nu=\sup_{x\in \Om} a(x),\;\;m=\inf_{\p\Om} h,\;\;\bar{m}=\sup_{\p\Om}h,\;\mbox{and}\;R=\mbox{diam}(\Om);\;\mbox{assume that}\;\;m>0\;\;\mbox{and}\;\;\nu<\infty.$$ 
Fix $y\in \p\Om$. Let $\vep>0$ such that $\vep<\bar{m}/2.$ Let $\eta_0>0$ be such that for $0<\eta\le \eta_0$,
\eqRef{sec7.1001}
h(y)-\vep\le h(x)\le h(y)+\vep,\;\;\forall x\in \overline{B}_\eta(y)\cap \p\Om.
\ee

Let $u$ solve (\ref{sec7.1710}) with $\dl=2\bar{m}$ and call $M=\sup_\Om u$. Set $\Om_\eta=\Om\setminus \overline{B_\eta(y)}.$
By Lemma \ref{sec8.7}, there is a unique solution $\phi>0$ of
\eqRef{sec7.174}
H(D\phi, D^2\phi)+\lam a(x) \phi^k=0,\;\;\mbox{in $\Om_\eta$, and $\phi=2\bar{m}$ on $\p\Om_\eta$.}
\ee
By Lemma \ref{sec4.13}, $u\ge 2\bar{m}$ on $\p\Om_\eta$. Using Lemma \ref{sec4.71}, (\ref{sec7.174}), Lemmas \ref{sec3.7} and \ref{sec4.13}, $2\bar{m}< \phi\le u\le M$, for any $\eta>0$, small.

{\bf Case (a):} Suppose that (\ref{sec2.450}) (i) holds and $\Om$ is any domain. 
Set $r=|x-y|$ and   
\eqRef{sec7.1740}
\bar{\phi}(x)=h(y)+\vep + (2\bar{m}-h(y)-\vep)\frac{r^\beta}{\eta^\beta},\;\;\mbox{in $\overline{B_\eta(y)\cap \Om}$,}
\ee
where $\beta=2-\bar{s}$, see (\ref{sec2.450}). Recalling (\ref{sec7.174}) and (\ref{sec7.1740}), we define
\eqRef{sec7.175}
w(x)=\left\{ \begin{array}{lcc} \phi(x), &&  \forall x\in \Om_\eta,\\  \bar{\phi}(x), && \forall x\in \overline{B_\eta(y)\cap \Om}. \end{array} \right.
\ee
Note that $w\in C(\overline{\Om})$, $w(y)=h(y)+\vep$, $w=2\bar{m}$, on $\p B_\eta(y)$, and $w\ge h$ on $\p\Om$. Our goal is to choose $0<\eta\le \eta_0$ (see (\ref{sec7.1001})) such that $w$ is a super-solution in $\Om$. 
By (\ref{sec2.61})(also see (\ref{sec7.101})) there is an $\eta_1\in (0, \eta_0]$ such that in $B_\eta(y)\cap \Om$, $\forall \eta\in (0, \eta_1]$, we have 
\bea\label{sec7.1751}
H(D\bar{\phi}, D^2\bar{\phi})+\lam a(x) \bar{\phi}^k&\le& \lam \nu \bar{\phi}^k-\left( \frac{2\bar{m}-h(y)-\vep}{\eta^\beta} \right)^k  \frac{\beta^k |m_2(\bar{s})|} { r^{\g-k\beta}}\nonumber \\
&\;\le&  \lam \nu (2\bar{m})^k- \left(2\bar{m}-h(y)-\vep\right)^k \frac{\beta^k|m_2(\bar{s})|} { \eta^{\g}}\le 0.
\eea
Here we have used $0<r<\eta$ and $\g-k\beta>0.$ We show that $w$ is a super-solution on
$\p B_\eta(y)\cap \Om$ implying that $w$ is a super-solution in $\Om$. Our idea is to choose a value of $\eta$ so that $d\bar{\phi}/dr(\eta)$ exceeds the radial rate of increase of $\phi$ on $r=\eta$. 

We now estimate $\phi$ (in $\Om_\eta$) on $\p B_\eta(y)$, from above. Note $2\bar{m}<\bar{\phi}\le M$ (see (\ref{sec7.174})). Set 
\eqRef{sec7.176}
\tht=\left( 1+ \frac{ 2(2M-2\bar{m})}{ 2\bar{m}-h(y)-\vep} \right)^{1/\beta},
\ee
where $\beta=2-\bar{s}$, see (\ref{sec2.450})(i). In $\eta\le r\le \tht\eta $, set
$$\psi(x)=2\bar{m}+d(r^\beta-\eta^\beta),\;\; \mbox{where}\;\;\;d=\frac{ 2M-2\bar{m}}{ \eta^\beta (\tht^\beta-1)}.$$
Clearly, $\psi=2\bar{m}$, on $r=\eta$, $\psi=2M$, on $r=\tht \eta$, and 
$2\bar{m}\le \psi\le 2M$. Using (\ref{sec7.176}),
\eqRef{sec7.177}
d= \frac{ 2M-2\bar{m}}{ \eta^\beta (\tht^\beta-1)}=\frac{2\bar{m}-h(y)-\vep}{2\eta^\beta}.
\ee
We choose $\eta_2\in (0,\eta_1]$ so that $\psi$ is a super-solution in $\eta<r<\tht \eta$, for 
any $\eta\in (0,\eta_2].$ This would then imply by Lemma \ref{sec4.71}, that $\phi\le \psi$, in $\eta<r<\tht \eta.$ 
Employing $\g-k\beta>0$,
(\ref{sec2.61}) and (\ref{sec7.177})
\ben
H(D\psi, D^2\psi)+\lam a(x) \psi^k &\le& \lam \nu (2M)^k-\frac{(d\beta)^k|m_2(\bar{s})|}{r^{\g-k\beta}}
\le \lam \nu (2M)^k-\frac{(d\eta^\beta \beta)^k|m_2(\bar{s})|}{\tht^{\g}\eta^{\g}}\\
&=&\lam \nu (2M)^k- \left(\frac{ 2M-2\bar{m}}{\tht^\beta-1}\right)^k \frac{\beta^k|m_2(\bar{s})|}{\tht^{\g}\eta^{\g}}\le 0,
\een
if $0<\eta\le \eta_2\le \eta_1\le \eta_0$ are small enough. Thus, (\ref{sec7.1001}) and (\ref{sec7.1751}) hold and this gives us the desired $\phi$, $\bar{\phi}$ and the upper bound $\psi$.

We show that $w$ is a super-solution. Let $\phi\in C^2$ be such that $w-\varphi$ has a minimum at some $p\in \p B_\eta(y)\cap \Om$. 
Then
(i) $\varphi(x)-\varphi(p)\le w(x)-w(p)=\bar{\phi}(x)-\bar{\phi}(p)\le 0,\;\forall x\in B_\eta(y)\cap \Om,$ and (ii) $\varphi(x)-\varphi(p)\le w(x)-w(p)\le \psi(x)-\psi(p),\;\forall x\in \Om_\eta.$
Using (\ref{sec7.1740}) and (\ref{sec7.177}), these yield,
$$\frac{\p \varphi}{\p r}(p)\ge \bar{\phi}^{\prime}(\eta-)=\beta \left( \frac{2\bar{m}-h(y)-\vep}{\eta} \right)\;\;\mbox{and}\;\;\frac{\p \varphi}{\p r}(p)\le \psi^{\prime}(\eta+)=\beta \left( \frac{2\bar{m}-h(y)-\vep}{2\eta} \right).$$
This is a contradiction and $w$ is a super-solution in $\Om$. Lemma \ref{sec4.71} and Remark \ref{sec2.8} imply the existence of a solution of (\ref{sec7.172}).

{\bf Case (b):} Assume that (\ref{sec2.450})(ii) holds and $\Om$ satisfies a uniform outer ball condition. Fix $y\in \p\Om$ and $\vep>0$, small.
Let $\rho>0$ and $z\in \IR^n$ be  such that 
$B_\rho(z)\subset \IR^n\setminus \Om$ and $y\in \p B_\rho(z)\cap \p\Om$. Set $r=|x-z|$. We modify (\ref{sec7.1001}) as follows. We choose $\eta_0>0$, small,
such that for $0<\eta\le \eta_0$,
$$h(y)-\vep\le h(x)\le h(y)+\vep,\;\;\mbox{$\forall x$ such that $\rho\le |x-z|\le \rho+\eta$}.$$
The ideas are similar to those in Case (a). Define $\Om_\eta=\Om\setminus \overline{B}_{\rho+\eta}(z)$ and note that $\Om_\eta$ satisfies a uniform outer ball condition. 
For $\eta\in (0, \eta_0]$, to be determined, there is a solution 
$\phi>0$ of
\eqRef{sec7.180}
H(D \phi, D^2\phi)+\lam a(x) \phi^k=0,\;\;\mbox{in $\Om_\eta$, and $\phi=2\bar{m}$ on $\p\Om_\eta$,}
\ee
by our hypothesis and Lemma \ref{sec8.7}. As noted in (\ref{sec7.174}), $2\bar{m}<\phi\le M$.

Call $A_\eta=\{x\in \Om:\; \rho<|x-z|<\rh+\eta\}$. We fix $\beta>\bar{s}-2$ and $s=\beta+2$, and define
\eqRef{sec7.181}
\bar{\phi}(x)=h(y)+\vep+ \left( \frac{ 2\bar{m}-h(y)-\vep}{\rho^{-\beta}-(\rho+\eta)^{-\beta} }\right)\left( \frac{1}{\rho^\beta}-\frac{1}{r^\beta} \right),\;\;\;\rho\le r\le \rho+\eta.
\ee
It is clear that $\bar{\phi}(y)=h(y)+\vep$, $\bar{\phi}=2\bar{m}$ on $r=\rho+\eta$, and $h(y)+\vep\le \bar{\phi}\le2\bar{m}.$  Using (\ref{sec7.180}) and (\ref{sec7.181}), we define
\eqRef{sec7.182}
w(x)=\left\{ \begin{array}{lcc} \phi(x), &&  \forall x\in \Om_\eta,\\  \bar{\phi}(x), && \forall x\in \overline{A}_\eta . \end{array} \right.
\ee
As done in Case (a), we will select $\eta>0$ such that $w$ is a super-solution in $\Om$. Clearly, $w(y)=h(y)+\vep$ and $w\ge h$ on $\p\Om$. Using  (\ref{sec2.62}) (see (\ref{sec7.102})) in $A_\eta$, i.e., $\rho\le r\le \rho+\eta$,
\ben
H(D\bar{\phi}, D^2\bar{\phi})+\lam a(x) \bar{\phi}^k&\le& \lam \nu (2\bar{m})^k-\frac{\beta^k|m_2(s)|}{r^{k \beta+\g}} 
\left( \frac{2\bar{m}-h(y)-\vep}{\rho^{-\beta}-(\rho+\eta)^{-\beta} }\right)^k\\
&\le &\lam \nu (2\bar{m})^k-\frac{\beta^k|m_2(s)|}{(\rho+\eta)^{k \beta+\g}} 
\left( \frac{2\bar{m}-h(y)-\vep}{\rho^{-\beta}-(\rho+\eta)^{-\beta} }\right)^k.
\een
Thus, $\bar{\phi}$ is a super-solution in $A_\eta$, if $\eta\in (0, \eta_1]$, for some $0<\eta_1\le \eta_0$, small. 

Next we choose $0<\eta_2\le \eta_1$ so that, for $0<\eta\le \eta_2$, the quantity
\eqRef{sec7.1820} 
K=\left( \frac{\rho^\beta}{ (\rho+\eta)^\beta-\rho^\beta} \right) \left( \frac{ 2\bar{m}-h(y)-\vep}{2(2M-2\bar{m})} \right)>1.
\ee
We calculate an upper bound for $\phi$ in $\Om_\eta\cap B_{\rho+\tht \eta}(z)$, where $\tht>1$ is to be determined. We take
\eqRef{sec7.183}
\psi(x)=2\bar{m}+\left( \frac{2M-2\bar{m}} { (\rho+\eta)^{-\beta}-(\rho+\tht \eta)^{-\beta}}\right)\left( \frac{1}{(\rho+\eta)^\beta}-\frac{1}{r^\beta}\right),\;\;\rho+\eta\le r\le \rho+\tht \eta.
\ee
Then $\psi=2\bar{m}$, on $r=\rho+\eta$, and $\psi=2M$, on $r=\rho+\tht \eta$. We calculate $\tht$ by requiring that $\bar{\phi}^{'}(\rho+\eta)>\psi^{'}(\rho+\eta)$, in particular, we impose that 
$$\left( \frac{2M-2\bar{m}} { (\rho+\eta)^{-\beta}-(\rho+\tht \eta)^{-\beta}}\right)=\frac{1}{2}\left( \frac{ 2\bar{m}-h(y)-\vep}{\rho^{-\beta}-(\rho+\eta)^{-\beta} }\right),$$
see (\ref{sec7.181}) and (\ref{sec7.182}). Rearranging and recalling (\ref{sec7.1820}),
$$
\frac{ (\rho+\tht \eta)^\beta}{ (\rho+\tht \eta)^\beta-(\rho+\eta)^\beta}=\left( \frac{\rho^\beta}{ (\rho+\eta)^\beta-\rho^\beta} \right) \left( \frac{ 2\bar{m}-h(y)-\vep}{2(2M-2\bar{m})} \right)=K.
$$
By (\ref{sec7.180}), $(\rho+\tht \eta)^\beta=K(\rho+\eta)^\beta/(K-1).$ Clearly, $\tht=\tht(\eta)>1$ if $0<\eta\le \eta_2$. 

We now show that $\psi$ is super-solution in $\rho+\eta<r<\rho+\tht \eta$, if $\eta$ is small enough. Using
(\ref{sec2.62})
\ben
H(D\psi, D^2\psi)+\lam a(x) \psi^k\le \lam \nu (2M)^k-\left( \frac{2M-2\bar{m}} { (\rho+\eta)^{-\beta}-(\rho+\tht \eta)^{-\beta}}\right)^k\frac{\beta^k |m_2(s)|}{(\rho+\tht\eta)^{k\beta+\g}}
\een
Thus, $\psi$ is a super-solution if $\eta\in(0,\eta_3]$, where $\eta_3\in (0,\eta_2]$, small. This determines $\tht$. 
As done in Case (a), $w$ is a super-solution in $\Om$. Lemma \ref{sec4.71} and Remark \ref{sec2.8} imply existence.  $\Box$

\section{Boundedness of $\lamo$: Proof of Theorem \ref{sec8.10}}

In this section, we show that $\lamo$ is bounded, see Theorem \ref{sec7.16}. In the first part of the proof we assume that (\ref{sec2.450})(i) holds and impose conditions $A,\;B$ and $C$. In the second part we address the case (\ref{sec2.450}) (ii) and conditions $A,\;B,\;C$ and $D$(radial symmetry) are assumed. 

Let $a(x)\in C(\Om)\cap L^{\infty}(\Om),\;\inf_\Om a>0,$ and $\dl>0$; we consider the problem
\eqRef{sec8.4}
H(Du, D^2u)+\lam a(x) |u|^{k-1}u=0,\;\;\mbox{in $\Om$}\;\;\mbox{and $u=\dl$ on $\p\Om$.}
\ee
We recall the definition of $\lamo$:
\eqRef{sec8.5}
\lamo=\sup\{\lam:\;\mbox{(\ref{sec8.4}) has a positive solution}\}.
\ee
By Theorem \ref{sec7.16}, if $0<\lam<\lamo$ in (\ref{sec8.4}) then $u\in C(\overline{\Om})$, $u>\dl,$ and $u$ is unique. 

\begin{rem}\label{sec8.51} Let $b\in C(\Om)$ with $a(x)\le b(x),\;\forall x\in \Om$. For some $\lam>0$, let $v\in C(\overline{\Om})$ solve
$$H(Dv, D^2v)+\lam b(x) |v|^{k-1}v=0,\;\;\mbox{in $\Om$, $v>0$ and $v=\dl$ on $\p\Om$}.$$
Then (\ref{sec8.4}) has a solution $u>0$ for $\lam$. Note that $v$ is super-solution of (\ref{sec8.4}), i.e., 
$H(Dv, D^2v)+\lam a(x) |v|^{k-1}v\le 0,$ in $\Om$, and $v=\dl$ on $\p\Om$. Also, $w=\dl$ is a sub-solution of (\ref{sec8.4}). By Lemma \ref{sec4.71} and Remark \ref{sec2.8}, there is solution $w\le u\le v$ of (\ref{sec8.4}) for $\lam.$
Let $\lamo(c)$ be the bound in (\ref{sec8.5}) for the weight function $c(x)$ then 
\eqRef{sec8.6}
\lamo(b)\le \lamo(a),\;\;\mbox{for $a(x)\le b(x),\;\forall x\in \Om$.} \;\;\;\;\Box
\ee
\end{rem}

We now prove Theorem \ref{sec8.10}.

{\bf Proof of Theorem \ref{sec8.10}.} Set $\mu=\inf_{x\in \Om} a(x)$ and assume that $\mu>0$.
Let $B_R(y)$, where $y\in \Om$, denote an in-ball of $\Om$. Call $B=B_R(y)$. By Lemma \ref{sec8.7}, if $0<\lam<\lamo$ then $0<\lam<\lam_B$.  

Consider the problem
\eqRef{sec8.12}
H(Dv, D^2v)+\lam \mu v^k=0,\;\mbox{in $B,$ and $v=\dl$ on $\p B$.}
\ee
By (\ref{sec8.6}), if (\ref{sec8.4}) has a solution on $B$ for some $\lam>0$ then (\ref{sec8.12}) has a positive solution since $\lam_B(\mu)\ge \lam_B(a).$
Thus, if we show that $\lam_B(\mu)<\infty$ then $\lamo(a)<\infty.$  We set 
$$\lam=\lam \mu\;\;\mbox{and}\;\; \lam_R=\mu \lam_B(\mu).$$
We assume that $\lam_R=\infty$ and derive a contradiction.

Let $\lam_1>0$ and recall that $k=k_1+k_2$ and $\g=k_1+2k_2$, see (\ref{sec2.42}).
Set $\lam_\ell=\ell^\g \lam_1,\;\ell=1,2,\cdots$. By Theorem \ref{sec7.17}, there is an $u_\ell\in C(\overline{B}),\;u_\ell>0,$ that solves
\eqRef{sec8.120}
H(Du_\ell, D^2u_\ell)+\lam_\ell u_\ell^k=0\;\mbox{in $B$, and $u_\ell=\dl_\ell$ on $\p B$, $\forall\;\ell=1,2,\cdots,$}
\ee
where $\dl_\ell>0$ is so chosen that $u_\ell(y)=1$. 

For each $\ell$, (\ref{sec8.12}) has a unique solution $v_\ell>0$ in $B$ with $v_\ell=1$ on $\p B$.
Thus by Lemma \ref{sec4.71}, $u_\ell/v_\ell\le \dl_\ell$ and $v_\ell/u_\ell\le 1/\dl_\ell$ implying $u_\ell=\dl_\ell v_\ell\ne 0.$ This shows that $\dl_\ell>0$.

{\bf Step 1:} We show that $\dl_\ell$ decreases to zero. Recall from (\ref{sec2.45}) that $\al=\g/k$. 
For $\forall \ell=1,2,\cdots,$ set $\tau_\ell=(\lam_\ell/\lam_{\ell+1})^{1/k}$, i.e,
$\tau_\ell=\left( \ell/(\ell+1) \right)^\al.$

Applying Remark \ref{sec4.12}(ii) and Lemma \ref{sec4.120},
$u_\ell(x) [u_{\ell+1}(x)]^{-\tau_\ell} \le \dl_\ell (\dl_{\ell+1})^{-\tau_\ell} .$
Taking $x=y$, we have $\dl^{\tau_\ell}_{\ell+1} \le \dl_\ell$ implying $(\dl_{\ell+1})^{(\ell+1)^{-\al}}\le (\dl_\ell)^{\ell^{-\al}}.$ Iterating 
\eqRef{sec8.121}
0<\dl_{\ell+1}\le (\dl_\ell)^{1/\tau_\ell}\le (\dl_{\ell-1})^{1/(\tau_\ell \tau_{\ell-1})} \le\cdots\le  (\dl_1)^{(\ell+1)^{\al}}.
\ee
Since $\dl_1<1$, the claim follows by letting $\ell\rightarrow \infty$. 

{\bf Step 2:} We employ scaling. For $\forall \ell=1,2,\cdots$, call $R_\ell=\ell R$ and $v_\ell(z)=u_\ell(x),\;\forall z\in B_{R_\ell}(y)$, where $z=y+\ell (x-y)$. Then
$H(Du_\ell, D^2u_\ell)=\ell^\g H(Dv_\ell, D^2v_\ell)$, $\lam_\ell=\lam_1\ell^\g$ and (\ref{sec8.120}) imply that
\eqRef{sec8.13}
H(Dv_\ell, D^2v_\ell)+\lam_1 v_\ell^k=0,\;\mbox{in $B_{R_\ell}(y),$ and $v_\ell(R_\ell)=\dl_\ell.$}
\ee
Moreover, $ v_\ell(y)=u_\ell(y)=1$. Call $B_\ell=B_{R_\ell}(y).$ 
\vsp
We address (\ref{sec2.450}) (i) and (ii) separately. Set $r=|x-y|$ and recall (\ref{sec2.61}) and (\ref{sec2.62}).

{\bf Case (a):} Suppose that (\ref{sec2.450})(i) holds; fix $\beta=2-\bar{s}$ where $0<\beta<1$. For $m$, large, take
$$w_m(x)=w_m(r)=\frac{1}{2}\left( 1- (1-\dl_m) \left(\frac{r}{R_m}\right)^\beta \right),\;\;\;0<r\le R_m.$$
Note that $w_m(y)=1/2$ and $w_m(R_m)=\dl_m/2$. Moreover, $w>0$ and the calculation in (\ref{sec2.61}) shows that 
$H(Dw_m, D^2w_m)+\lam_1 w_m^k\ge 0$, in $B_{R_m}(y)\setminus\{y\}$. Applying Lemma \ref{sec4.71} to $w_m$ and $v_m$ in 
$B_m\setminus\{y\}$, we see that
\eqRef{sec8.130}
w_m(x)\le v_m(x)/2,\;\;\;\forall x\in B_m(y).
\ee

{\bf Step a(i) :} Let $1<\ell\le m$. Applying Lemma \ref{sec4.71}, we see that 
$$\frac{v_\ell(x)}{v_m(x)}\le \frac{\dl_\ell}{\inf_{\{r=R_\ell\}} v_m}\;\;\;\mbox{and}\;\;\;\frac{v_m(x)}{v_\ell(x)}\le \frac{\sup_{\{r=R_\ell\}} v_m}{\dl_\ell},\;     \forall x\in B_\ell.$$ 
Taking $x=y$ in the above inequality, we get $\inf_{\{r=R_\ell\}} v_m\le \dl_\ell\le \sup_{\{r=R_\ell\}} v_m.$ Using (\ref{sec8.130}), we obtain that for each $\ell\le m$,
$$w_m(R_\ell)\le \dl_\ell/2.$$ 

{\bf Step a(ii) :} Taking $m=2\ell$, recalling that $R_{2\ell}=2R_\ell$, using Step a(i) and Step 2, we obtain that
\ben
\frac{1}{2}\left(1- \frac{1-\dl_{2\ell}}{ 2^\beta} \right)=w_{2\ell}(R_\ell)\le\frac{\inf_{\{r=R_\ell\}}v_{2\ell}}{2}\le  \frac{\dl_\ell}{2}.
\een
Letting $\ell\rightarrow \infty$, we obtain a contradiction to (\ref{sec8.121}).

{\bf Case (b):} Assume that (\ref{sec2.450})(ii) holds. Fix $\beta=s-2$, where $s>\bar{s}$. We impose conditions $A,\;B,\;C$ and $D$. See Remark \ref{sec8.14} in this context.
By condition $D$, $H$ is invariant under rotations.

{\bf Step b(i):} Using condition D, Lemma \ref{sec4.71} implies that $v_\ell$ is radial, see Step 2. By Lemma \ref{sec4.13}, $v_\ell(r)\ge v_\ell(\rho)$, for $0\le r\le \rho\le R_\ell$. Thus, $v_\ell(r)$ is non-increasing and $\sup v_\ell=v_\ell(y)=1$.

{\bf Step b(ii):} Let $1\le \ell\le m$, applying Lemma \ref{sec4.71} in $B_\ell$ we obtain that
$$1=\frac{v_m(o)}{v_\ell(o)}\le  \frac{v_m(R_\ell)}{\dl_\ell}\;\;\mbox{and}\;\;1=\frac{v_\ell(o)}{v_m(o)}\le  \frac{\dl_\ell}{v_m(R_\ell)}.$$ 
Clearly, $v_m(R_\ell)=\dl_\ell.$ Lemma \ref{sec4.71} implies that $v_m=v_\ell,$ in $B_\ell$. Thus, $v_m$ extends $v_\ell$ to $B_m$.

{\bf Step b(iii):} We claim that the decay estimate in Step 1 can not hold, leading to a contradiction and thus proving that $\lamo<\infty$. We proceed as follows.
By Step b(i), there is a $0<\rho<R$ such that for any $\ell$, $v_\ell>1/2$, in $B_\rho(y)$. Consider the function
$$\omega(x)=\frac{1}{2} \left(\dl_{2\ell}+(1-\dl_{2\ell}) \frac{r^{-\beta}-R_{2\ell}^{-\beta} }{\rho^{-\beta}-R_{2\ell}^{-\beta} }\right),\;\;\forall\;\rho\le r\le R_{2\ell}.$$
Using (\ref{sec2.62}), $H(D\omega, D^2\omega)+\lam \nu \omega^k\ge 0$, in $\rho<r<R_{2\ell}$. Since $\omega(\rho)=1/2$ and $\omega(R_{2\ell})=\dl_{2\ell}/2$, applying Lemma \ref{sec4.71} in $B_{R_{2\ell}}\setminus B_\rho(y)$ we see that $\omega(x)\le v_{2\ell}(x)$, in $\rho\le r\le R_{2\ell}$. Recall Steps 1, 2 and Step b(ii), and take $r=R_\ell$ to find 
$$\frac{C}{\ell^\beta}  \le \omega(R_\ell)\le v_{2\ell}(R_\ell)=\dl_\ell\le (\dl_1)^{\ell^\al},$$
where $C>0$, depends on $\beta, \rho$ and $R$. Letting $\ell\rightarrow \infty$, we obtain a contradiction. $\Box$

\begin{rem}\label{sec8.14} In Case (b), condition D is not required if non-negative super-solutions $w$ i.e, $H(Dw, D^2w)\le 0$, satisfy a Harnack inequality i.e,
$\inf_{B_s(z)} w\ge C \sup_{B_s(z)} w$, where $C$ is a universal constant and $B_{4s}(z)\subset B$. In Step b(iii), $\forall \ell$, $v_\ell(x)\ge C,\;\forall x\in B_\rho(y)$, with $0<\rho\le R/8$. Apply Steps a(i), (ii) and b(iii) to get a proof. Also, the proof works if $w$ satisfies a modulus of continuity depending on $\sup w$. $\Box$
\end{rem}

\section{Existence of a positive first eigenfunction}

This section has two sub-sections. In Sub-section I, we show the existence of a positive eigenfunction on a general domain when (\ref{sec2.450})(i) holds and conditions $A, \; B$ and $C$ apply. In Sub-section II, we discuss (\ref{sec2.450})(ii), impose conditions $A,\;B,\; C$ and $D$ and take $\Om$ to be a ball. 

{\bf Sub-section I:} Let $\Om\subset \IR^n$ be any bounded domain and assume that (\ref{sec2.450}) (i) holds. Fix $\beta=2-\bar{s},\;0<\beta<1,$ recall (\ref{sec2.61}) and take $a\in C(\Om)\cap L^{\infty}(\Om),\;\inf_\Om a>0.$ See \cite{AMJ, B, BL, J}.

\begin{lem}\label{sec9.01}{(The Harnack inequality and H$\ddot{o}$lder Continuity)} Let $w\in lsc(\Om)\cap L^{\infty}(\Om),\;w\ge 0,$ solve
$H(Dw, D^2w)+\lam a(x) |w|^{k-1}w\le 0,\;\;\mbox{in $\Om$.}$
For any $y\in \Om$ and $R>0$ such that $B_{4R}(y)\subset \Om$, we have a universal constant $C>0$ such that
$$\sup_{B_R(y)} w\le C \inf_{B_R(y)} w,\;\;\;\mbox{and}\;\;\;|w(x)-w(z)|\le (3R)^{-\beta}(\sup_{B_R(y)} w)|x-z|^\beta,\;\forall x,z \in B_R(y).$$
 \end{lem} 
 \begin{proof} Let $w(y)>0$, for some $y\in \Om$, and $B_{4R}(y)\subset \Om$. Set $A=B_{4R}(y)\setminus\{y\}$, $r=|x-y|$ and 
$$\psi(x)=w(y)\left( 1- r^\beta(4R)^{-\beta} \right),\;\;\mbox{in $B_{4R}(y)$.}$$
Thus, $\psi(y)=w(y)$, $\psi=0$, on $\p B_{4R}(y)$ and by (\ref{sec2.61}), $H(D\psi, D^2\psi)>0$ in $A$. 
Clearly, on $\p A$, $\inf(w-\psi)=0$. If $\inf_A(w-\psi)<0$ and $p\in A$ is a point of minimum then $H(D\psi(p), D^2\psi(p))+\lam a(p) |w(p)|^{k-1}w(z)> 0$, which is a contradiction .Thus, $w(x)\ge \psi(x)>0,$ in $B_{4R}(y)$.

Observing that for any $z\in B_R(y)$, $B_R(y)\subset B_{2R}(z)$ and arguing as above,
\eqRef{sec9.02}
w(x)\ge w(z) \left(1-|x-z|^{\beta}(3R)^{-\beta} \right)\;\;\mbox{for any $x,\;z\in B_{R}(y)$.}
\ee
Since $|x-z|\le 2R$, the claim holds. To show the H$\ddot{\mbox{o}}$lder continuity of $w$, we write
(\ref{sec9.02}) as
$$w(z)-w(x)\le w(z) |x-z|^{\beta}(3R)^{-\beta}\le (3R)^{-\beta}(\sup_{B_R(y)} w) |x-z|^\beta.$$
Taking $x\in B_R(y)$ and replacing $x$ by $z$, we get the claim.
\end{proof}

{\bf Proof of Theorem \ref{sec9.00}(i).} Let $\dl>0$ and $\lam_\ell,\;\ell=1,2,\cdots$ be an increasing sequence such that $\lim_{\ell\rightarrow \infty}\lam_\ell=\lamo$. For each $\ell$, there is a unique positive
$u_\ell\in C(\overline{\Om})$ such that
$$H(Du_\ell, D^2u_\ell)+\lam_\ell a(x) u_\ell^k=0,\;\;\mbox{in $\Om$, and $u_\ell =\dl$, on $\p\Om$.}$$
Set $\tht_\ell=\sup_\Om u_\ell$. By Theorem \ref{sec7.17}, $\lim_{\ell\rightarrow \infty}\tht_\ell=\infty$. Calling $v_\ell=u_\ell/\tht_\ell$, we see that $\sup v_\ell=1$ and
$v_\ell|_{\p\Om}\rightarrow 0.$ By Lemma \ref{sec9.01}, there is a sub-sequence $v_m$ and a function $v\in C(\Om)$ such that $v_m$ converges uniformly to $v$ on compact subsets. Thus, 
$$v\ge 0,\;\;\;\sup_\Om v=1\;\;\;\mbox{and} \;\;\;\lim_{m\rightarrow \infty}v_m=\lim_{m\rightarrow \infty}\frac{\dl}{\tht_m}=0\;\;\mbox{on $\p\Om$.}$$ 

(a) We show that $H(Dv, D^2v)+\lamo a(x) v^k=0\;\mbox{in $\Om$.}$ Let $\phi\in C^2$ and $p\in \Om$ be such that $v-\phi$ has a minimum at $p$. Set $B_\vep=B_\vep(p)$, for $\vep>0$, small. Let $\hat{\Om}$ be a compact sub-domain of $\Om$ containing $B_\vep$. Set $d=$dist$(p, \p\hat{\Om})$, $k=\max(3, 3d^{-4})$ and $\psi=\phi(x)-k|x-p|^4$. 
Take $m$, large, so that $\sup_\Om |v_m-v|\le \vep^4$ on $\hat{\Om}$. Then 
$$v_m(x)-\psi(x)\ge k|x-p|^4+ (v_m-v)(x)+(v-v_m)(p)+(v_m-\psi)(p).$$   
Noting that $v_m-\psi>(v_m-\psi)(p)$, on $\Om\setminus B_\vep$, $v_m-\psi$ has a minimum at some $p_m\in B_\vep$, and 
$H(D\psi(p_m), D^2\psi(p_m))+\lam_m a(p_m) v_m(p_m)^k\le 0$. Letting $\vep\rightarrow 0$, $H(D\phi(p), D^2\phi(p))+\lamo a(p) v(p)^k\le 0$. The proof that $v$ is a sub-solution is similar.

(b) We show that $v\in C(\overline{\Om})$. Let $y\in \p\Om$ and set $r=|x-y|$. Recalling (\ref{sec2.450}) (i) and (\ref{sec2.61}), take 
$w_\vep(x)=\vep+(1-\vep) (r/\rho)^\beta$, in $B_\rho(y)$,  where 
$0<\vep<1/4$ is small and $\rho$ is to be determined. Set $\nu=\sup_\Om a$ and recall that $\g-k\beta>0$. There is a $\rho>0$, small and independent of $\vep$ so that 
\ben
H(Dw_\vep, D^2w_\vep)+\lam a(x) w_\vep^k\le\lam \nu-\frac{(\beta (1-\vep))^k|m_2(\bar{s})|}{\rho^{k\beta} r^{\g-k\beta}}\le \lam \nu-\left(\frac{3\beta}{4}\right)^k \frac{|m_2(\bar{s})|}{\rho^\g}< 0,\;0<r\le \rho.
\een
Also, $\vep\le w_\vep\le 1$, and, for large $m$, $w_\vep\ge v_m$ on $\p(\Om\cap B_\rho(y))$. By Lemma \ref{sec4.71}, $w_\vep\ge v_m$, in $B_\rho(y)\cap \Om$. Letting 
$m\rightarrow \infty$, $w_\vep\ge v$, in $\Om\cap B_\rho(y)$. Clearly, $0\le \liminf_{x\rightarrow y}v(x)\le \limsup_{x\rightarrow y} v(x)\le w_\vep(y)=\vep.$ Since $\vep$ is arbitrary, $v(y)=0$ and 
$v\in C(\overline{\Om})$. 

(c) Next,we show that $v>0$ in $\Om$. Let $p\in \Om$ be such that $v(p)=1$. Recall from (b) the bound $v\le w=(r/\rho)^\beta$, in $B_\rho(y)\cap \Om$. Clearly, if we take $r=\rho/2$, we have $v<1$. Thus, $p$ is at least $\rho/2$ away from $\p\Om$. We now apply Harnack's inequality in Lemma \ref{sec9.01} to conclude $v>0$ in $\Om$.
$\Box$
\vsp
{\bf Sub-section II:} Assume that (\ref{sec2.450})(ii) holds. We take $\Om$ to be the ball $B_R(o)$, where $R>0$, and prove the existence of a positive radial first eigenfunction. 
Set $\lam_R=\lam_{B_R(o)}$, $r=|x|,\;\forall x\in \IR^n$ and take $a(x)=1,\forall x\in \Om$. We impose conditions $A,\;B,\;C$ and $D$. As observed in Step b(i) of Theorem \ref{sec8.10} there is a $v\in C(\overline{B_R(o)}),\;v>0,$ radial and non-increasing in $r$, that solves
\eqRef{sec9.1}
H(Dv, D^2v)+\lam v^k=0,\;\;\mbox{in $B_R(o)$, $v=\dl$ on $\p B_R(o)$, and $0<\lam<\lam_R$}.
\ee

{\bf Scaling property.} For $0<R_1<R_2$, set $B_i=B_{R_i}(o),\;i=1,2$. By Lemma \ref{sec8.7}, $\lam_{R_1}\ge \lam_{R_2}$. Let $u$ be the solution of
(\ref{sec9.1}) in $B_1$, for some $0<\lam<\lam_{R_1}.$ Define $v(y)=u(x),\;\forall y\in B_2$, where $y=sx,$ where $s=R_2/R_1$. As $\g=k_1+2k_2$, we have
$H(Du, D^2u)=s^{\g} H( Dv, D^2v)$ and
$$H(Dv, D^2 v)+s^{-\g}\lam v^k=0,\;\;\mbox{in $B_2,$ and $v=\dl,$ on $\p B_2$.}$$
By (\ref{sec8.5}) and Theorem \ref{sec7.17} (iv), $\lam_{R_1}=s^\g \lam_{R_2}$
\eqRef{sec9.3}
\lam_{R_1} R_1^\g=\lam_{R_2} R_2^\g.
\ee

We now show existence of a first eigenfunction on a ball. Set $r=|x|,\;\forall x\in \IR^n$. Call $B=B_R(o)$.

{\bf Step 1:} Fix $\dl>0$, $0<\lam<\lam_R$ and let $u=u(r)\in C(\overline{B}),\;u>0,$ be the unique solution of
$$H(Du, D^2u)+\lam u^k=0,\;\mbox{in $B$, $u(R)=\dl$ and $u(o)=1$.}$$
Set $\hat{R}=(\lam_R/\lam)^{1/\g} R$. By (\ref{sec9.3}), $\lam_{\hat{R}}=\lam$, $\hat{R}>R$ and $\lam$ is the first eigenvalue of $H$ on $B_{\hat{R}}(o)$.  
\vsp
{\bf Step 2:} Let $\{\lam_\ell\}_{\ell=1}^\infty$ be a decreasing sequence such that $\lam<\lam_\ell<\lam_R$ and $\lim_{\ell \rightarrow \infty} \lam_\ell= \lam.$ Call $R_\ell$ such that
$\lam_\ell R_\ell^\g=\lam_R R^\g$ and $B_\ell=B_{R_\ell}(o)$. By (\ref{sec9.3}), $\lam_\ell=\lam_{B_\ell}$ is the first eigenvalue of $H$ on $B_{\ell}$. By Step 1,
$$R<R_1<\cdots<R_{\ell}<\cdots<\hat{R},\;\;\mbox{and}\;\;\lim_{\ell\rightarrow \infty}R_\ell=\hat{R}.$$
Since $\forall \ell$, $\lam<\lam_\ell,$ there is a unique $u_\ell\in C(\overline{B_\ell}),\;u_\ell>0$, radial and non-increasing such that
\eqRef{sec9.40}
 H(Du_\ell, D^2u_\ell)+\lam u_\ell^k=0,\;\;\mbox{in $B_\ell,$ and $u_\ell(R_\ell)=\dl_\ell$,}
 \ee
where $\dl_\ell>0$ is so chosen that $u_\ell(o)=1$, see (\ref{sec9.1}).
Since $\lam_m\ge \lam_\ell$ and $R_m\le R_\ell,$ for $m \le \ell$, using Lemma \ref{sec4.71} in $B_{m}$, we see that 
$$1= \frac{u_{m}(o)}{u_{\ell}(o)}\le \frac{u_{m}(R_{m})}{u_\ell(R_m)}\;\;\mbox{and}\;\;1= \frac{u_{\ell}(o)}{u_{m}(o)}\le \frac{u_{\ell}(R_{m})}{u_m(R_m)}.$$
Hence, $u_m(R_m)=u_\ell(R_m)=\dl_m$ and $u_m=u_\ell$ in $B_m$. Thus, $u_\ell$ extends $u_m$ to $B_\ell$, in particular, $u_\ell$ extends $u_0$ to $B_\ell$. Moreover, $\dl_\ell$ is decreasing. 

{\bf Step 3:} We claim that $ \lim_{\ell\rightarrow \infty} \dl_\ell=0.$ Suppose not. Since $\dl_\ell$ is decreasing, $\forall \ell,\;\dl_\ell\ge \eta$, for some $\eta>0$.
Clearly, $u_\ell\ge \eta$. Take $s=1/2$ in the estimate in Theorem \ref{sec7.15}(i) to see that there is a solution $v$ to
$$H(Dv, D^2v)+\tilde{\lam} v^k=0,\;\mbox{in $B_\ell$ and $v(R_\ell)=\dl_\ell$, where $\tilde{\lam}=\lam+\vep$ and} \;0<\vep\le \frac{\lam k \eta}{2(1-\tht \eta/2)}.$$
Here, we may choose $\lam\left(1+ k\eta/2 \right)<\tilde{\lam}<\lam_\ell$.
Since $\eta$ is independent of $\ell$, letting $\ell\rightarrow \infty$ and using Step 2, we obtain a contradiction.

{\bf Step 4:} Recalling Step 2, for $x\in B_{\hat{R}}(o)$, define $u(x)=\lim_{\ell \rightarrow \infty} u_\ell(x).$
By (\ref{sec9.40}), $u\in C(B_{\hat{R}}(o))$ and $H(Du, D^2u)+\lam u^k=0,\;\mbox{in $B_{\hat{R}}(o)$.}$
Since $u$ is radial and decreasing, define $u(\hat{R})=0$. We have $u\in C(\overline{B_{\hat{R}}})$, since $u(R_\ell)=u_\ell(R_\ell)=\dl_\ell\rightarrow 0$, see Step 3.

{\bf Step 5:} We scale $u$ as follows. Set $w(\rho)=u(r)$ where $\rho=rR/\hat{R}.$ Thus, $w\in C(\overline{B}),\;w>0,$ solves
$H(Dw, D^2w)+\lam (\hat{R}/R)^{\g} w^k=0,$ in $B_R(o),$ and $w(R)=0$.
By Step 1, $\lam_R=\lam (\hat{R}/R)^{\g}$ and, thus, $w$ is a first eigenfunction on $B_R(o)$. $\Box$

\vsp

Department of Mathematics, Western Kentucky University, Bowling Green, Ky 42101, USA

Department of Liberal Arts, Savannah College of Arts and Design, Savannah, GA 31405, USA

\end{document}